\documentclass{article}
\usepackage{array}
\usepackage{tabularx}
\usepackage{amsmath,amssymb}
\usepackage{mathrsfs}
\usepackage{amsthm}
\usepackage{comment}
\usepackage[all]{xy}
\usepackage[
pdftex,
bookmarks=false,
colorlinks=true,
debug=true,
naturalnames=true,
pdfnewwindow=true]{hyperref}

\theoremstyle{definition}
\newtheorem{Def}{Definition}[section]
\newtheorem{Thm}[Def]{Theorem}

\newtheorem{Prop}[Def]{Proposition}
\newtheorem{Cor}[Def]{Corollary}
\newtheorem{Rem}[Def]{Remark}

\numberwithin{equation}{section}

\newcommand{\C}{\mathbb{C}}
\newcommand{\kk}{\Bbbk}
\newcommand{\Cc}{\mathcal{C}}

\newcommand{\g}{\mathfrak{g}}

\newcommand{\Z}{\mathbb{Z}}
\newcommand{\modcat}{\text{-}\mathrm{mod}}
\newcommand{\modfd}{\text{-}\mathrm{mod}_{\mathrm{fd}}}
\newcommand{\vdim}{\underline{\dim}\,}
\newcommand{\Ext}{\mathop{\mathrm{Ext}}\nolimits}
\newcommand{\End}{\mathop{\mathrm{End}}\nolimits}
\newcommand{\Hom}{\mathop{\mathrm{Hom}}\nolimits}
\newcommand{\Aut}{\mathop{\mathrm{Aut}}\nolimits}

\newcommand{\Ker}{\mathop{\text{Ker}}}
\newcommand{\Ima}{\mathop{\text{Im}}}

\newcommand{\Q}{\mathbb{Q}}
\newcommand{\Stab}{\mathop{\mathrm{Stab}}\nolimits}

\newcommand{\Rep}{\mathop{\mathrm{Rep}}\nolimits}
\newcommand{\KP}{\mathrm{KP}}
\newcommand{\pprime}{\prime \prime}
\newcommand{\M}{\mathfrak{M}}
\newcommand{\LL}{\mathfrak{L}}
\newcommand{\Mg}{\mathfrak{M}^{\bullet}}
\newcommand{\Lg}{\mathfrak{L}^{\bullet}}
\newcommand{\G}{\mathbb{G}}
\newcommand{\hlam}{{\boldsymbol \lambda}}

\newcommand{\cP}{\mathsf{P}}
\newcommand{\cQ}{\mathsf{Q}}
\newcommand{\lP}{\mathscr{P}}
\newcommand{\cl}{\mathsf{cl}}

\newcommand{\cR}{\mathsf{R}}

\newcommand{\hK}{\widehat{K}}

\newcommand{\SG}{\mathfrak{S}}

\newcommand{\bW}{\mathbb{W}}
\newcommand{\hR}{\widehat{R}}
\newcommand{\hH}{\widehat{H}}
\newcommand{\HH}{\mathbb{H}}
\newcommand{\tZ}{\widetilde{\mathfrak{Z}}}
\newcommand{\B}{\mathcal{B}}
\newcommand{\F}{\mathcal{F}}
\newcommand{\ii}{\mathbf{i}}
\newcommand{\mm}{\mathbf{m}}
\newcommand{\Td}{\mathrm{Td}}
\newcommand{\ch}{\mathrm{ch}}
\newcommand{\RR}{\mathrm{RR}}
\newcommand{\rr}{\mathfrak{r}}
\newcommand{\bV}{\mathbb{V}}
\newcommand{\hV}{\widehat{V}}
\newcommand{\Zg}{Z^{\bullet}}
\newcommand{\cZ}{\mathcal{Z}}
\newcommand{\hO}{\widehat{\mathcal{O}}}
\newcommand{\bK}{\mathbb{K}}
\newcommand{\AH}{H^{\mathrm{af}}}

\numberwithin{equation}{section}

\title{Geometric realization of Dynkin quiver type
quantum affine Schur-Weyl duality}
\author{Ryo Fujita\thanks{Department
of Mathematics, Kyoto University, 
Oiwake Kita-Shirakawa Sakyo Kyoto 
606-8502 JAPAN, 
\texttt{E-mail:rfujita@math.kyoto-u.ac.jp}}}

\begin{document}

\maketitle


\begin{abstract}
For a Dynkin quiver $Q$ of type $\mathsf{ADE}$
and a sum $\beta$ of simple roots,
we construct a bimodule
over the quantum loop algebra and the quiver Hecke algebra
of the corresponding type via equivariant $K$-theory,
imitating Ginzburg-Reshetikhin-Vasserot\rq{}s 
geometric realization
of the quantum affine Schur-Weyl duality.
Our construction is based on
Hernandez-Leclerc\rq{}s isomorphism between 
a certain graded quiver variety and the space of representations of 
the quiver $Q$ of dimension vector $\beta$.
We identify the functor induced from our bimodule
with Kang-Kashiwara-Kim's  generalized 
quantum affine Schur-Weyl duality functor. 
As a by-product, we verify a conjecture by
Kang-Kashiwara-Kim on the simpleness
of some poles of normalized $R$-matrices for any quiver $Q$
of type $\mathsf{ADE}$.
\end{abstract}


\section{Introduction}

For a fixed pair $(n, d)$ of positive integers,
we have the following two fundamental objects:
the complex simple Lie algebra $\mathfrak{sl}_{n+1}$ of type $\mathsf{A}_{n}$
and 
the symmetric group $\SG_{d}$ of degree $d$.
The natural $(\mathfrak{sl}_{n+1}, \SG_{d})$-bimodule structure
on the tensor power $(\C^{n+1})^{\otimes d}$  
produces a close relationship between 
their representation theories.
This phenomenon is
known as the classical Schur-Weyl duality 
and has many interesting variants.

The quantum affine Schur-Weyl duality 
is a variant
involving their quantum affinizations:
the quantum loop algebra 
$U_{q}(L\mathfrak{sl}_{n+1})$ of $\mathfrak{sl}_{n+1}$
and 
the affine Hecke algebra $\AH_d(q)$ of $GL_d$.
Both algebras are defined over $\kk := \Q(q)$.
Here
we equip the tensor power
$\bV^{\otimes d}$ of 
the natural representation $\bV := \kk^{n+1}[z^{\pm 1}]$ of
$U_{q}(L \mathfrak{sl}_{n+1})$
with a commuting right action of 
$\AH_{d}(q^{2})$ using the $R$-matrices. 
Chari-Pressley \cite{CP96} proved 
that the induced functor
$$
\AH_{d}(q^{2}) \modcat
\to
U_{q}(L \mathfrak{sl}_{n+1}) \modcat;
\quad
M \mapsto \bV^{\otimes d} \otimes_{\AH_{d}(q^{2})} M   
$$ 
gives an equivalence between suitable subcategories of finite-dimensional modules.

The quantum affine Schur-Weyl duality 
has a beautiful geometric realization
due to Ginzburg-Reshetikhin-Vasserot \cite{GRV94}.
Here we recall their construction briefly.
Let 
$\mu_{d} : \F_{d} \to \mathcal{N}_{d}$
be the Springer resolution of
the nilpotent cone $\mathcal{N}_{d}$
of $\mathfrak{gl}_{d}(\C)$,
where $\F_{d}$ is the cotangent bundle of 
the full flag variety of $GL_{d}(\C)$.
The morphism $\mu_{d}$
is equivariant with respect to 
a natural action of the group
$\G_{d} := GL_{d}(\C) \times \C^{\times}$,
where $\C^{\times}$ acts as the scalar multiplication
on the cone $\mathcal{N}_{d}$.
Due to Ginzburg and Kazhdan-Lusztig,
the affine Hecke algebra $\AH_{d}(q^{2})$
is isomorphic to the convolution algebra
$K^{\G_{d}}(\cZ_{d}) \otimes_{A}
\kk$
of the equivariant $K$-group of the
Steinberg variety $
\cZ_{d} := \F_{d} \times_{\mathcal{N}_{d}} \F_{d}$,
where $A = R(\C^{\times}) = \Z[v^{\pm 1}]$
is the representation ring of $\C^{\times}$ and
$- \otimes_{A} \kk$ means the specialization $v \mapsto q$.
On the other hand, 
we consider another Steinberg type variety
$Z_{d} := \M_{d} \times_{\mathcal{N}_{d}} \M_{d}$.
Here $\M_d$ is the cotangent bundle 
of the variety of partial flags in $\C^{d}$
of length $\le n+1$.
Due to Ginzburg-Vasserot,
there is an algebra homomorphism
$\Phi : U_{q}(L \mathfrak{sl}_{n+1}) 
\to K^{\G_{d}}(Z_{d}) \otimes_{A}\kk$
with some good properties.
Based on these facts, Ginzburg-Reshetikhin-Vasserot 
considered the intermediary fiber product
$\M_{d} \times_{\mathcal{N}_{d}}
\F_{d}$
and identified its equivariant $K$-group with the bimodule 
$\bV^{\otimes d}$.
More precisely, they established an isomorphism 
$\bV^{\otimes d} 
\cong
K^{\G_{d}}(\M_{d} \times_{\mathcal{N}_{d}} 
\F_{d}) \otimes_{A} \kk$
making
the following diagram commute:
$$
\xy
\xymatrix{
U_{q}(L\mathfrak{sl}_{n+1}) 
\ar[r]
\ar[d]^{\Phi}
&
\End\left( \bV^{\otimes d} \right)
\ar[d]^-{\cong}
&
\AH_{d}(q^{2})
\ar[l]
\ar[d]^-{\cong}
\\
K^{\G_{d}}(Z_{d})\otimes_{A} \kk
\ar[r]
&
\End \left( K^{\G_{d}}(\M_{d} \times_{\mathcal{N}_{d}} 
\F_{d}) \otimes_{A} \kk
\right)
&
K^{\G_{d}}(\cZ_{d})\otimes_{A}\kk,
\ar[l]
}
\endxy
$$
where horizontal arrows denote the bimodule structures.

Recently, 
in a series 
of papers 
\cite{KKK18, KKK15, KKKO15, KKKO16},
Kang, Kashiwara, Kim and Oh 
established 
some interesting generalized versions
of the quantum affine Schur-Weyl duality.
One of them (treated in \cite{KKK15} by Kang-Kashiwara-Kim) is associated with
a pair $(Q, \beta)$ of 
a Dynkin quiver $Q$ of type $\mathsf{ADE}$
and a sum $\beta=\sum_{i} d_{i} \alpha_{i}$ of simple roots,
which plays a similar role as the pair $(n, d)$ in the previous paragraphs.
One player is
the quantum loop algebra $U_{q}(L\g)$ 
of the complex simple Lie algebra $\g$
whose Dynkin diagram is the underlying graph of $Q$.
The other is 
the quiver Hecke (KLR) algebra $H_{Q}(\beta)$
associated with $(Q, \beta)$, or actually its completion $\hH_{Q}(\beta)$ 
along the grading.
The quiver Hecke algebra $H_{Q}(\beta)$ is regarded as a generalization 
of the affine Hecke algebra $\AH_{d}(q)$ 
from the viewpoint of the categorification of the quantum group. 
Inspired by the work of Hernandez-Leclerc \cite{HL15},
Kang-Kashiwara-Kim \cite{KKK15}
constructed on a left $U_{q}(L\g)$-module $\hV^{\otimes \beta}$
which is a direct sum of some tensor products of affinized fundamental modules
a commuting right action of the algebra $\hH_{Q}(\beta)$ by
using the normalized $R$-matrices. 
However, to make the $\hH_{Q}(\beta)$-action well-defined, 
we need to assume the simpleness of some poles of the normalized $R$-matrices.
This assumption was verified for type $\mathsf{AD}$ in \cite{KKK15} by 
an explicit computation of
the denominators of the normalized $R$-matrices.
On the other hand, for type $\mathsf{E}$,
this remains a conjecture. 
Under this well-definedness assumption, 
Kang-Kashiwara-Kim \cite{KKK15} also proved that the induced functor 
$$
\hH_{Q}(\beta) \modfd \to U_{q}(L \g)\modfd; \quad 
M \mapsto \hV^{\otimes \beta} \otimes_{\hH_{Q}(\beta)} M
$$
is exact, 
factors through the $\beta$-block $\Cc_{Q, \beta}$ of a monoidal 
full subcategory 
$\Cc_{Q}$ of $U_{q}(L\g)\modfd$ introduced by Hernandez-Leclerc \cite{HL15}
and gives a bijection between 
the simple isomorphism classes.
More recently, the author \cite{Fujita17} proved that
it actually gives an equivalence 
$\hH_{Q}(\beta) \modfd \simeq \Cc_{Q, \beta}$ 
by using the notion of affine highest weight category.
Note that here we forget the gradings by working with the completion 
$\hH_{Q}(\beta)$.   

In the present paper, we give a geometric realization of
the bimodule $\hV^{\otimes \beta}$ imitating Ginzburg-Reshetikhin-Vasserot\rq{}s realization.
In our case, the nilpotent cone $\mathcal{N}_{d}$ is replaced by
the space $E_{\beta}$ of representations of the quiver $Q$ over $\C$
of dimension vector $\beta$.
The group 
$G_{\beta}:= \prod_{i} GL_{d_{i}}(\C)$ naturally acts on $E_{\beta}$.
Instead of the Springer resolution 
$\F_{d} \to \mathcal{N}_d$, we consider a proper morphism $\F_{\beta} \to E_{\beta}$
from a ``quiver flag variety'' $\F_{\beta}$
introduced by Lusztig
in order to construct the canonical basis of the quantum group.
Varagnolo-Vasserot \cite{VV11} proved that 
the quiver Hecke algebra $H_{Q}(\beta)$ is isomorphic to the convolution algebra of the
equivariant Borel-Moore homology $H^{G_{\beta}}_{*}(\cZ_{\beta}, \kk)$,
where $\cZ_{\beta} := \F_{\beta} \times_{E_{\beta}} \F_{\beta}$.
After 
completion, 
it is isomorphic to the completed equivariant 
$K$-group $\hK^{G_{\beta}}(\cZ_{\beta})_{\kk}$.
On the $U_{q}(L\g)$-side,
we consider a canonical $G_{\beta}$-equivariant proper morphism
$\Mg_{\beta} \to \Mg_{0, \beta}$ between 
certain graded quiver varieties.
By Nakajima \cite{Nakajima01},
we have an algebra homomorphism
$
\widehat{\Phi}_{\beta} : U_{q}(L \g) \to \hK^{G_{\beta}}(\Zg_{\beta})_{\kk},
$
where $\Zg_{\beta} := \Mg_{\beta} \times_{\Mg_{0, \beta}} \Mg_{\beta}$.
The key of our construction is a $G_{\beta}$-equivariant 
isomorphism 
$\Mg_{0, \beta} \cong
E_{\beta}$ due to Hernandez-Leclerc \cite{HL15}, 
which was originally established in order to give a geometric interpretation 
to their isomorphism between the Grothendieck ring $K(\Cc_{Q})$ 
and the coordinate ring of the maximal unipotent subgroup 
(see Remark \ref{Rem:Grotisom}).
This allows us to form the intermediary fiber product
$\Mg_{\beta} \times_{E_{\beta}} \F_{\beta}$.

\begin{Thm}[=Theorem~\ref{Thm:left} + 
\ref{Thm:main}, see also Remark~\ref{Rem:KKK}]
There is an isomorphism
$$\hV^{\otimes \beta} \cong \hK^{G_{\beta}}(\Mg_{\beta} \times_{E_{\beta}} \F_{\beta})_{\kk}$$
such that the following diagram commutes (up to a twist):
$$
\xy
\xymatrix{
U_{q}(L\g) 
\ar[r]
\ar[d]^-{\widehat{\Phi}_{\beta}}
&
\End\left(\hV^{\otimes \beta}\right)
\ar[d]^-{\cong}
&
\hH_{Q}(\beta)
\ar[l]
\ar[d]^-{\cong}
\\
\hK^{G_{\beta}}(\Zg_{\beta})_{\kk}
\ar[r]
&
\End \left( \hK^{G_{\beta}}(\Mg_{\beta} \times_{E_{\beta}} 
\F_{\beta})_{\kk}
\right)
&
\hK^{G_{\beta}}(\cZ_{\beta})_{\kk},
\ar[l]
}
\endxy
$$
where the horizontal arrows denote the bimodule structures.
\end{Thm}

Actually,
our geometric construction of the $\hH_{Q}( \beta )$-action 
is independent of that of \cite{KKK15}, 
which shares the same characterization of the actions. 
Therefore, their comparison yields a uniform proof of:

\begin{Thm}[=Corollary \ref{Cor:KKKconj}]
\label{Thm:KKKconj}
Kang-Kashiwara-Kim's conjecture 
 \cite[Conjecture 4.3.2]{KKK15}
on the simpleness
of some specific poles of normalized $R$-matrices 
for tensor products of fundamental modules
is true for 
any quiver $Q$ of type $\mathsf{ADE}$.
\end{Thm}

Besides, a discussion involving geometric extension algebras 
yields another proof of the equivalence 
$\hH_{Q}(\beta) \modfd \simeq \Cc_{Q, \beta}$
given by the bimodule $\hV^{\otimes \beta}$
without using affine highest weight categories
(Theorem~\ref{Thm:equiv}).                  
We would also remark that 
we do not use any results of \cite{KKK18}, \cite{KKK15} 
for our proofs.  

The present paper is organized as follows.
In Section \ref{Sec:HL}, 
we recall the definition of graded quiver varieties 
$\Mg_{\beta}$ and $\Mg_{0, \beta}$, and  
Hernandez-Leclerc\rq{}s
isomorphism $\Mg_{0, \beta} \cong E_{\beta}$.
In Section \ref{Sec:Conv}, we study 
the convolution algebra $\hK^{G_{\beta}}(\cZ_{\beta})_{\kk}$ 
(resp.~$\hK^{G_{\beta}}(\Zg_{\beta})_{\kk}$) and 
recall its relation to the quiver Hecke algebra $H_{Q}(\beta)$
(resp.~the quantum loop algebra $U_{q}(L\g)$).
In the final section \ref{Sec:SW}, we 
study the structure of the bimodule
$\hK^{G_{\beta}}(\Mg_{\beta} \times_{E_{\beta}} \F_{\beta})_{\kk}$.

While the author 
was writing this paper, 
there appeared a preprint by 
Oh-Scrimshaw \cite{OS18} in arXiv
which also proves Theorem~\ref{Thm:KKKconj} 
by a different approach. They
compute the denominators 
of the normalized $R$-matrices for type $\mathsf{E}$ 
explicitly with a computer.

\paragraph{Acknowledgment.}
The author 
thanks Syu Kato for helpful discussions and comments.
He also thanks Ryosuke Kodera
for suggesting him to study the geometric realization of
the generalized 
quantum affine Schur-Weyl duality. 


\paragraph{Convention.}	
An algebra $A$ is associative and unital.
We denote by $A^{\mathrm{op}}$ (resp.~$A^{\times}$)
the opposite algebra 
(resp.~the set of invertible elements)
of $A$ and
by $A \text{-}\mathrm{mod}$
the category of left $A$-modules.
Working over a base field $\mathbb{F}$, 
the symbol $\otimes$ (resp.~$\Hom$)
stands for $\otimes_{\mathbb{F}}$ (resp.~$\Hom_{\mathbb{F}}$) 
if there is no other clarification.
If $A$ is an $\mathbb{F}$-algebra,
we denote by $A \modfd$ the category
of finite-dimensional left $A$-modules.


\section{Hernandez-Leclerc\rq{}s isomorphism}
\label{Sec:HL}

\subsection{Notation}
\label{Ssec:Notation}
  
Throughout this paper,
we fix a finite-dimensional complex simple Lie algebra
$\g$ of type $\mathsf{ADE}$ and
a quiver $Q=(I, \Omega)$
whose underlying graph is the Dynkin diagram of $\g$,
where $I=\{ 1,2, \ldots,n \}$
(resp. $\Omega$) is the set of vertices (resp. arrows).
For an arrow $h \in \Omega$,
let  $h^{\prime}, h^{\pprime} \in I$ denote its origin and goal
respectively.
We write $i \sim j$ (resp.~$i \to j$) if there is an arrow 
$h \in \Omega$
such that $\{i,j\}=\{h^{\prime}, h^{\pprime}\}$
(resp.~$(i, j) = (h^{\prime}, h^{\pprime})$).
Then the Cartan matrix $(a_{ij})_{i,j \in I}$ of $\g$ is given by 
$$
a_{ij} = \begin{cases}
2 & \text{if $i=j$}; \\
-1 & \text{if $i \sim j$}; \\
0 & \text{otherwise}.
\end{cases}
$$
Let $\cP^{\vee} = \bigoplus_{i \in I} \Z h_{i}$ 
be the coroot lattice of $\g$.
The fundamental weights $\{ \varpi_{i}\}_{i \in I}$ 
form a basis of  the weight lattice 
$\cP = \Hom_{\Z}(\cP^{\vee}, \Z)$
which is dual to $\{h_{i} \}_{i \in I}$.
Let $\alpha_{i} = \sum_{j \in I} a_{ij} \varpi_{j}$ 
be the $i$-th simple root and 
$\cQ = \bigoplus_{i \in I} \Z \alpha_{i} \subset \cP$ 
be the root lattice.
We put $\cP^{+} = \sum_{i \in I}\Z_{\ge 0} \varpi_{i}$
and $\cQ^{+} = \sum_{i \in I} \Z_{\ge 0} \alpha_{i}$.
The Weyl group is the finite group $W$ of 
linear transformations on $\cP$ generated by 
the set $\{r_{i}\}_{ i \in I}$ of simple reflections, 
which are given by
$r_{i}(\lambda) := \lambda - \lambda(h_{i}) \alpha_{i}$
for $\lambda \in \cP$. 
The set $\cR^{+}$ of positive roots is defined by 
$\cR^{+} = (W\{ \alpha_{i}\}_{ i \in I}) \cap \cQ^{+}$.


\subsection{Representations of Dynkin quiver}
\label{Ssec:Rep_quiver}

For an element $\beta \in \cQ^{+}$,
we fix an $I$-graded $\C$-vector space 
$D=\bigoplus_{i \in I} D_{i}$ such that 
$\vdim D := \sum_{i \in I} (\dim D_{i}) \alpha_{i} = \beta$.
Let us consider the space 
$$
E_{\beta} := \bigoplus_{h \in \Omega}
\Hom(D_{h^{\prime}}, D_{h^{\pprime}}) 
$$
of representations of the quiver $Q$ 
of dimension vector $\beta$. 
On the space $E_{\beta}$,
the group $G_{\beta} := \prod_{i \in I}\mathop{GL}(D_{i})$
acts by conjugation. 
The set $G_{\beta} \backslash E_{\beta}$ of $G_{\beta}$-orbits
is  
naturally in bijection with the set of
isomorphism classes of representations of 
the quiver $Q$ of 
dimension vector $\beta$.  
By Gabriel\rq{}s theorem, 
for each $\alpha \in \cR^{+}$ 
there exists an indecomposable representation 
$M_\alpha$ such that $\vdim M_\alpha = \alpha$
uniquely up to isomorphism.
The correspondence
$\alpha \mapsto M_{\alpha}$
gives a bijection between the set $\cR^{+}$ 
of positive roots and
the set of isomorphism classes of indecomposable objects 
of the category $\Rep Q$ of 
finite-dimensional representations of $Q$.
Hence, the set
$$
\KP(\beta) :=
\left\{ (m_{\alpha}) 
\in (\Z_{\ge 0})^{\mathsf{R}^{+}}
\; \middle| \; \textstyle
\sum_{\alpha \in \cR^{+}} m_{\alpha} \alpha
 = \beta \right\}
$$
of Kostant partitions of $\beta$ labels 
the set of $G_{\beta}$-orbits:
$G_{\beta} \backslash E_{\beta} 
= \{ \mathbb{O}_{\mm} \}_{\mm \in \KP(\beta)}$,
where for each $\mm = (m_\alpha) \in \KP(\beta)$, 
the $G_{\beta}$-orbit $\mathbb{O}_{\mm}$ corresponds
to the isomorphism class of the representation
$\bigoplus_{\alpha \in \cR^{+}} M_{\alpha}^{\oplus m_{\alpha}}$.
We have the natural $G_{\beta}$-orbit stratification
\begin{equation}
\label{Eq:decE}
E_{\beta} = \bigsqcup_{\mm \in \KP(\beta)}
\mathbb{O}_{\mm}.
\end{equation}


\subsection{Repetition quiver}
\label{Ssec:Repet_quiver}

We fix a height function  
$\xi : I \to \Z ; i \mapsto \xi_{i}$ of 
the quiver $Q$ i.e.~it satisfies 
$\xi_{i} = \xi_{j} +1$ if $i \to j$.
Such a function $\xi$ is determined up to adding a constant.
Choose a total ordering $I=\{i_{1}, i_{2}, \ldots , i_{n} \}$
such that $\xi_{i_1} \ge \xi_{i_2} \ge \cdots \ge \xi_{i_n}$
and define the corresponding Coxeter element
$c := r_{i_1}r_{i_2} \cdots r_{i_n} \in W$.

The repetition quiver $\widehat{Q} = (\widehat{I}, \widehat{\Omega})$
is an infinite quiver 
defined by
\begin{align*}
\widehat{I} &:= \{(i,p) \in I \times \Z \mid p-\xi_{i} \in 2\Z \}, \\
\widehat{\Omega} &:= \{(i,p) \to (j, p+1) \mid 
(i, p), (j, p+1) \in \widehat{I}, \; i \sim j 
\}.
\end{align*}
It is well-known (cf.~\cite{Happel88}) that 
there exists an isomorphism 
$\phi$ from
the Auslander-Reiten quiver of the derived category 
$D^{b}(\Rep Q)$ to the repetition quiver 
$\widehat{Q}$, 
which depends on the choice of $\xi$
and is described as follows.
Since each indecomposable object of $D^{b}(\Rep Q)$ 
is isomorphic to a unique stalk complex $M_{\alpha}[k]$ 
for some $(\alpha, k) \in \cR^{+} \times \Z$,
we have a bijection between the sets of vertices 
$$
\cR^{+} \times \Z \ni (\alpha, k) \mapsto 
\phi(M_\alpha [k]) \in \widehat{I},
$$
which we denote by the same symbol $\phi$.
This bijection $\phi : \cR^{+} \times \Z \to \widehat{I}$
is determined inductively as follows:
\begin{itemize}
\item
For each $i \in I$, we put 
$\gamma_{i} := \sum_{j} \alpha_{j}$ where 
$j$ runs all the vertices $j \in I$ such that there is a path
in $Q$ from $j$ to $i$.
Then $M_{\gamma_{i}}$ is an injective hull 
of the $1$-dimensional representation $M_{\alpha_{i}}$.
We define $\phi (\gamma_{i}, 0) := (i, \xi_{i})$;
\item
Inductively, if $\phi(\alpha, k) = (i, p)$ 
for $(\alpha, k) \in \mathsf{R}^{+} \times \Z$, then we define as:
\begin{align*}
\phi(c^{\pm1}(\alpha), k) &:= (i, p\mp 2) 
& \text{if $c^{\pm 1}(\alpha) \in \mathsf{R}^{+}$}, \\
\phi(-c^{\pm1}(\alpha), k \mp 1) &:= (i, p\mp 2) 
& \text{if $c^{\pm 1}(\alpha) \in - \mathsf{R}^{+}$}. 
\end{align*}
\end{itemize}
In the followings, we only consider the restriction of 
the bijection $\phi$ on $\cR^{+} = \cR^{+} \times \{0\}$,
which we denote by the same symbol, i.e. we define
$\phi(\alpha) := \phi(\alpha, 0)$ for $\alpha \in \cR^{+}$.


\subsection{Graded quiver varieties}
\label{Ssec:Quiver_var}

In this subsection, 
we recall the definition of the graded quiver varieties. 
A basic reference is \cite{Nakajima01}.

For elements
$\nu = \sum_{i \in I} n_{i} \alpha_{i} \in \cQ^{+}$
and 
$\lambda = \sum_{i \in I} l_{i} \varpi_{i}
\in \cP^{+}$,
we fix $I$-graded $\C$-vector spaces
$
V = \bigoplus_{i \in I} V_{i},
W = \bigoplus_{i \in I} W_{i}
$
such that $\dim V_{i} = n_{i}, 
\dim W_{i} = l_{i}$ for each $i \in I$.
We form the following space of linear maps:
$$
\mathbf{M}(\nu, \lambda) :=
\left(
\bigoplus_{i \sim j} \Hom (V_{j}, V_{i}) \right)
\oplus
\left( \bigoplus_{i \in I} \Hom (W_{i}, V_{i}) \right)
\oplus
\left(
\bigoplus_{i \in I} \Hom (V_{i}, W_{i}) \right)
$$
On the $\C$-vector space $\mathbf{M}(\nu, \lambda)$,
the groups
$G(\nu) := \prod_{i \in I} \mathop{GL}(V_{i})$,
$G(\lambda) := \prod_{i \in I} \mathop{GL}(W_{i})$ 
act by conjugation and the $1$-dimensional torus
$\C^{\times}$ acts by the scalar multiplication. 
We write an element of $\mathbf{M}(\nu, \lambda)$ 
as a triple
$(B,a,b)$ of linear maps $B = \bigoplus B_{ij}$,
$a = \bigoplus a_{i}$ and $b = \bigoplus b_{i}$.
Let $\mu = \bigoplus_{i \in I} \mu_{i} 
: \mathbf{M}(\nu, \lambda) 
\to \bigoplus_{i \in I} \mathfrak{gl}(V_{i})$
be the map given by 
$$
\mu_{i}(B, a, b)
= a_{i} b_{i}
+\sum_{j \sim i} \varepsilon(i,j) B_{ij}B_{ji},
$$
where $\varepsilon(i,j) := 1$ (resp.~$-1$)
if $j \to i$ (resp.~$i \to j$).
A point $(B, a, b) \in \mu^{-1}(0)$ is said to be stable
if there exists no non-zero $I$-graded subspace 
$V^{\prime} \subset V$ such that 
$B(V^{\prime}) \subset V^{\prime}$ and 
$V^{\prime} \subset \Ker b$.
Let $\mu^{-1}(0)^{\mathrm{st}}$ be 
the set of stable points, on which $G(\nu)$ acts freely. 
Then we consider a set-theoretic quotient
$$\M(\nu, \lambda) := \mu^{-1}(0)^{\mathrm{st}} / G(\nu).$$
It is known that this quotient has a structure of
a non-singular quasi-projective variety which 
is isomorphic to a quotient in the geometric invariant theory.
We also consider the affine algebro-geometric
quotient
$$\M_{0}(\nu, \lambda) := \mu^{-1}(0)/\!/G(\nu) 
= \mathrm{Spec}\, \C[\mu^{-1}(0)]^{G(\nu)},$$
together with a canonical projective morphism
$\M(\nu, \lambda) 
\to \M_{0}(\nu, \lambda)$.
These quotients $\M(\nu, \lambda)$, $\M_{0}(\nu, \lambda)$
naturally inherit the actions of the group 
$\G(\lambda) := G(\lambda) \times \C^{\times}$,
which makes the canonical projective morphism into
a $\G(\lambda)$-equivariant morphism.  

For $\nu, \nu^{\prime} \in \cQ^{+}$ such that
$\nu^{\prime} - \nu \in \cQ^{+}$,  
there is a natural closed embedding
$\M_{0} (\nu , \lambda)
\hookrightarrow 
\M_{0}(\nu^{\prime},\lambda).
$
With respect to these embeddings, 
the family $\{ \M_{0}(\nu, \lambda)\}_{\nu \in \cQ^{+}}$
forms an inductive system, 
which stabilizes at some $\nu \in \cQ^{+}$.
We consider the union (inductive limit) and 
obtain the following combined 
$\G(\lambda)$-equivariant morphism: 
$$
\pi : \M(\lambda) 
:= \bigsqcup_{\nu} \M(\nu, \lambda)
\to 
\M_{0}(\lambda) 
:= \bigcup_{\nu} \M_{0}(\nu, \lambda).
$$  
We denote the fiber $\pi^{-1}(0)$ of the origin 
$0 \in \M_{0}(\lambda)$ by $\LL(\lambda)
= \bigsqcup_{\nu \in \cQ^{+}} \LL(\nu, \lambda)$.
Note that $\M(0,\lambda) = \LL(0, \lambda)$
consists of a single point.

Next we consider a free abelian monoid 
$\lP^{+} = \mathbb{Z}_{\ge 0} \widehat{I}$
with the free generating set $\widehat{I}$.
Define a homomorphism 
$\mathsf{cl} : \lP^{+} \to \cP^{+}$
by $\cl(i, p)=\varpi_{i}$.
For an element $\hlam = \sum
 l_{i, p} (i,p) \in \lP^{+}$ 
with $\cl(\hlam) = \lambda$, we fix a decomposition
$W_{i} = \bigoplus_{p} W_{i, p}$ such that
$\dim W_{i, p} = l_{i,p}$ for each $(i,p) \in \widehat{I}$.
Define a group homomorphism
$f_{i} : \C^{\times} \to
\prod_{p} GL(W_{i, p})
\subset 
GL(W_{i})$
by 
$f_{i} (t) |_{W_{i, p}} 
:= t^{p} \cdot \mathrm{id}_{W_{i, p}}$
for $t \in \C^{\times}$.
We put $T(\hlam) := (\prod_{i \in I}f_{i} \times \mathrm{id})
(\C^{\times})
\subset \G(\lambda)$
and consider the subvarieties of $T(\hlam)$-fixed points:
$$
\pi^{\bullet} := \pi^{T(\hlam)} : \Mg(\hlam) := \M(\lambda)^{T(\hlam)} 
\to \Mg_{0}(\hlam) := \M_{0}(\lambda)^{T(\hlam)}.
$$
We refer these varieties as the graded quiver varieties.
We put $\Lg(\hlam) := \LL(\lambda)^{T(\hlam)}
= (\pi^{\bullet})^{-1}(0)$.

The centralizer of $T(\hlam)$ inside $\G(\lambda)$
is 
$$\G(\hlam) \equiv
G(\hlam) \times \C^{\times} := 
\prod _{(i,p) \in \widehat{I}} \mathop{GL}(W_{i, p}) \times \C^{\times}
\subset \G(\lambda),$$ which naturally acts 
on the varieties $\Mg(\hlam)$, $\Mg_{0}(\hlam)$, $\Lg(\hlam)$. 
The morphism $\pi^{\bullet}$ is $\G(\hlam)$-equivariant. 


\subsection{Hernandez-Leclerc\rq{}s isomorphism}
\label{Ssec:HL}

Let $\lP_{0}^{+}
\subset \lP^{+}$ be the submonoid 
generated by the subset 
$\phi(\cR^{+}) \subset \widehat{I}$.
For an element 
$\beta := \sum_{i \in I} d_{i} \alpha_{i} \in \cQ^{+}$,
we define 
$\hlam_{\beta} := \sum_{i \in I} d_{i} \phi(\alpha_{i})
\in \lP^{+}_{0}$. 
In this case, 
we write $\pi_{\beta} : \Mg_{\beta} \to \Mg_{0, \beta}$
instead of $\pi^{\bullet} : \Mg(\hlam_{\beta}) \to 
\Mg_{0}(\hlam_{\beta})$ for simplicity.
For each $i \in I$,
we identify the vector space $D_{i}$ in 
Subsection \ref{Ssec:Rep_quiver} 
with the vector space $W_{\phi(\alpha_{i})}$
in Subsection \ref{Ssec:Quiver_var}. 
This induces the identification $G_{\beta} = G(\hlam_{\beta})$.
We write $\G_{\beta}$, $T_{\beta}$
instead of $\G(\hlam_\beta)$, $T(\hlam_{\beta})$ respectively. 
By the inclusion
$G_{\beta} = G_{\beta} \times \{ 1\}
\subset G_{\beta} \times \C^{\times} = \G_{\beta}$,
the group $G_{\beta}$ is regarded as a subgroup of 
the group $\G_{\beta}$. 
Then the
multiplication map $G_{\beta} \times T_{\beta} \to \G_{\beta}$
gives an isomorphism of algebraic groups 
\begin{equation}
\label{Eq:group}
G_{\beta} \times T_{\beta} \cong \G_{\beta}.
\end{equation} 
We equip an action of the group $\G_{\beta}$
on the space $E_{\beta}$ via the projection
$\G_{\beta} \cong G_{\beta} \times T_{\beta} 
\twoheadrightarrow G_{\beta}.$

\begin{Thm}[Hernandez-Leclerc~\cite{HL15} Theorem~9.11]
\label{Thm:HL}
There exists
a $\G_{\beta}$-equivariant
isomorphism of varieties
$$
\M_{0, \beta}^{\bullet}
\xrightarrow{\cong}
E_{\beta}.$$
\end{Thm}

Henceforth, we identify the graded quiver variety  
$\Mg_{0, \beta}$ with the space $E_{\beta}$
under the isomorphism in Theorem~\ref{Thm:HL}.

We recall some properties of fibers of the $\G_{\beta}$-equivariant morphism 
$\pi_{\beta}:\Mg_{\beta} \to E_{\beta}$.
By the injective map 
$$\KP(\beta) \ni (m_{\alpha}) \mapsto 
\sum_{\alpha} m_{\alpha} \phi(\alpha) \in \lP^{+}_{0}, $$
we regard $\KP(\beta)$ as a subset of $\lP^{+}_{0}$.
Then we have a disjoint union decomposition
$$
\lP^{+}_{0} = \bigsqcup_{\beta \in \cQ^{+}} \KP(\beta).
$$

\begin{Prop}[cf.~\cite{Fujita17} Section 3]
\label{Prop:fiber}
Let $\mm \in \KP(\beta)$ and 
pick a point
$x_{\mm} \in \mathbb{O}_{\mm}$.
\begin{enumerate}
\item \label{Prop:fiber:isom}
We have an isomorphism 
$\pi_{\beta}^{-1}(x_{\mm}) \cong \Lg(\mm)$.
\item \label{Prop:fiber:stab}
The maximal reductive quotient 
of the stabilizer $\Stab_{G_{\beta}}(x_{\mm})
\subset G_{\beta}$
of the point $x_{\mm}$ is isomorphic to 
$G(\mm)$.
\item 
The isomorphism in (\ref{Prop:fiber:isom}) induces 
the following commutative diagram:
$$
\xy
\xymatrix{
\Aut(\pi_{\beta}^{-1}(x_{\mm})) 
\ar[r]^-{\cong}
& 
\Aut(\Lg(\mm)) 
\\
\Stab_{G_\beta}(x_{\mm})
\ar[u]
\ar@{->>}[r]
&
G(\mm),
\ar[u]
}
\endxy
$$
where the vertical arrows are the action maps and the
lower horizontal arrow is the canonical quotient map in 
(\ref{Prop:fiber:stab}).  
\end{enumerate}
\end{Prop}
\section{Convolution and geometric extension algebras}
\label{Sec:Conv}

Let $\kk$ be a field of characteristic zero.
Later in Subsection \ref{Ssec:Nakajima}, we specialize $\kk = \Q(q)$.

\subsection{Preliminary on equivariant geometry}
\label{Ssec:Pre}

For the materials in this subsection,
we refer \cite{CG97} and \cite{EG00}.  

Let $G$ be a complex linear algebraic group.
A $G$-variety $X$ is a quasi-projective 
complex algebraic variety
equipped with an algebraic action of the group $G$.  
We set $\mathrm{pt} := \mathop{\mathrm{Spec}} \C$ 
with the trivial $G$-action.
The equivariant $K$-group $K^{G}(X)$ is 
defined to be the Grothendieck group 
of the abelian category
of $G$-equivariant coherent sheaves on $X$
which is a module over the
representation ring
$R(G)=K^{G}(\mathrm{pt})$.
We put
$$K^{G}(X)_{\kk} := K^{G}(X)\otimes_{\Z}\kk, \quad
R(G)_{\kk} := R(G)\otimes_{\Z}\kk.$$
Let $I \subset R(G)_{\kk}$ be the augmentation ideal,
i.e.~the ideal 
generated by 
virtual representations of dimension $0$.
We define the $I$-adic completions by
$$
\hK^{G}(X)_{\kk} := 
\varprojlim_{k} K^{G}(X)_{\kk}/I^{k} K^{G}(X)_{\kk}, 
\quad
\hR(G)_{\kk}:= \varprojlim_{k} R(G)_{\kk} / I^{k}.
$$ 
The completed 
$K$-group $\hK^{G}(X)_{\kk}$
is a module over the algebra $\hR(G)_{\kk}$.

On the other hand, 
the $G$-equivariant Borel-Moore homology
with $\kk$-coefficients
$$
H^{G}_{*}(X, \kk) = \bigoplus_{k \in \Z}H_{k}^{G}(X, \kk),
$$
is a module over the $G$-equivariant 
cohomology ring 
$H_{G}^{*}(\mathrm{pt}, \kk)$
of $\mathrm{pt}$ (with the cup product).
Let us define the completion of
a $\Z$-graded $\kk$-vector space
$V=\bigoplus_{k \in \Z}V_{k}$
by 
$
V^{\wedge} := \prod_{k \in \Z} V_{k}.
$
The completion $H^{*}_{G}(\mathrm{pt}, \kk)^{\wedge}$
naturally becomes a $\kk$-algebra and 
the completion $H^{G}_{*}(X, \kk)^{\wedge}$
becomes 
a module over $H^{*}_{G}(\mathrm{pt}, \kk)^{\wedge}$.  

Assume that our $G$-variety $X$ is a 
$G$-stable
closed subvariety of 
a non-singular ambient 
$G$-variety $M$.
Then we have the 
$G$-equivariant local Chern character map
$$
(\ch^{G})^{M}_{X}
: \hK^{G}(X)_{\kk} \to H^{G}_{*}(X, \kk)^{\wedge}.
$$
relative to $M$.
We simply write $\ch^{G}$ instead of $(\ch^{G})^{M}_{X}$
if the pair $(M, X)$ is obvious
from the context.
When $X=M=\mathrm{pt}$, the corresponding
Chern character map induces an
isomorphism
of $\kk$-algebras
$$
\hR(G)_{\kk}
=\hK^{G}(\mathrm{pt})_{\kk}
\cong H^{G}_{*}(\mathrm{pt}, \kk)^{\wedge}
= H_{G}^{*}(\mathrm{pt}, \kk)^{\wedge}.
$$
We identify
$
H^{G}_{*}(\mathrm{pt}, \kk)^{\wedge}
$
with
$\hR(G)_{\kk}$ 
via this isomorphism.
Then $(\ch^{G})^{M}_{X}$ is regarded as
an $\hR(G)_{\kk}$-homomorphism.

For a $G$-equivariant vector bundle $E$ on 
a non-singular $M$,
let $\Td^{G}(E) \in H^{*}_{G}(M, \kk)^{\wedge}$ be the
$G$-equivariant Todd class. 
This is an invertible element with respect to 
the cup product.
For the tangent bundle $T_{M}$ of $M$,
we put $\Td^{G}_{M} := \Td^{G}(T_{M})$.

\begin{Thm}[Equivariant Riemann-Roch \cite{EG00}]
\label{Thm:EG}
For $i=1,2$, let $X_{i}$ be a $G$-variety
which is a $G$-stable closed subvariety of 
a non-singular ambient $G$-variety $M_{i}$.
Assume that a $G$-equivariant morphism
$\tilde{f}:M_{1}\to M_{2}$ restricts to a proper morphism
$f : X_{1} \to X_{2}$. 
Then we have
$$
f_{*} \left(\Td_{M_{1}}^{G} \cdot
(\ch^{G})_{X_{1}}^{M_{1}}(\zeta) \right)
= \Td^{G}_{M_{2}} \cdot
(\ch^{G})_{X_{2}}^{M_{2}}(f_{*} \zeta),
\quad
\zeta \in \hK^{G}(X_{1})_{\kk}.$$
\end{Thm}

The following proposition is standard.

\begin{Prop}
\label{Prop:chex}
Let $M$ be a non-singular $G$-variety.
Let $Y \subset X \subset M$ be $G$-stable closed subvarieties,
and $i : Y \hookrightarrow X$, 
$j : X\setminus Y \hookrightarrow X$
be inclusions.
Then 
we have the following commutative diagram:
$$
\xy
\xymatrix{
\hK^{G}(Y)_{\kk}
\ar[r]^-{i_{*}}
\ar[d]^-{(\ch^{G})^{M}_{Y}}
& 
\hK^{G}(X)_{\kk}
\ar[r]^-{j^{*}} 
\ar[d]^-{(\ch^{G})^{M}_{X}}
&
\hK^{G}(X\setminus Y)_{\kk}
\ar[d]^-{(\ch^{G})^{M\setminus Y}_{X\setminus Y}}
\\
H_{*}^{G}(Y, \kk)^{\wedge}
\ar[r]^-{i_{*}}
&
H_{*}^{G}(X, \kk)^{\wedge}
\ar[r]^-{j^{*}}
&
H_{*}^{G}(X\setminus Y, \kk)^{\wedge}.
}
\endxy
$$
\end{Prop}

Next we consider the convolution products.
Let $M_i$ be non-singular $G$-varieties for $i=1,2,3$.
We denote by $p_{ij} : M_1 \times M_2 \times M_3
\to M_i \times M_j$ the projection to the $(i,j)$-factors
for $(i,j) = (1,2), (2,3), (1,3)$.
Let $Z_{12} \subset M_{1} \times M_{2}$ and
$Z_{23} \subset M_{2} \times M_{3}$ be 
$G$-stable closed subvarieties such that
the morphism 
$$
p_{13} : p_{12}^{-1}(Z_{12}) \cap p_{23}^{-1}(Z_{23}) 
\to Z_{13} 
:= p_{13}(p_{12}^{-1}(Z_{12}) \cap p_{23}^{-1}(Z_{23}) )
$$
is proper.
Then we define the convolution product
$
* : K^{G}(Z_{12}) \otimes_{R(G)} K^{G}(Z_{23}) 
\to K^{G}(Z_{13})
$
relative to $M_{1} \times M_{2} \times M_{3}$ by
$$
\zeta * \eta := p_{13*} (p_{12}^{*}\zeta
\otimes^{\mathbb{L}}_{M_{1} \times M_{2} \times M_{3}} p_{23}^{*}\eta ),
\quad
\zeta \in K^{G}(Z_{12}), \eta \in K^{G}(Z_{23}).
$$
This naturally induces the convolution product
on the completed $G$-equivariant $K$-groups
$\hK^{G}(Z_{12})_{\kk} \otimes_{\hR(G)_{\kk}}
\hK^{G}(Z_{23})_{\kk} \to \hK^{G}(Z_{13})_{\kk}
$.
Similarly, we have the convolution product 
on the $G$-equivariant Borel-Moore homologies
$H_{*}^{G}(Z_{12}, \kk) \otimes_{H^{*}_{G}(\mathrm{pt}, \kk)}
H_{*}^{G}(Z_{23}, \kk) \to H^{G}_{*}(Z_{13}, \kk)$
relative to $M_{1} \times M_{2} \times M_{3}$
and its completed version
$H_{*}^{G}(Z_{12}, \kk)^{\wedge} \otimes_{\hR(G)_{\kk}}
H_{*}^{G}(Z_{23}, \kk)^{\wedge} 
\to H_{*}^{G}(Z_{13}, \kk)^{\wedge}.
$ 

Under the situation in the previous paragraph,
for each $(i,j)= (1,2), (2,3), (1,3)$, 
we also define the $G$-equivariant
Riemann-Roch homomorphism
$\RR^{G} : \hK^{G}(Z_{ij})_{\kk} \to 
H_{*}^{G}(Z_{ij}, \kk)^{\wedge}$
relative to $M_{i} \times M_{j}$
by 
$$
\RR^{G}(\zeta) := 
(p_{i}^{*} \Td^{G}_{M_{i}}) \cdot 
(\ch^{G})_{Z_{ij}}^{M_{i} \times M_{j}} (\zeta), \quad
\zeta \in \hK^{G}(Z_{ij})_{\kk},
$$
where $p_{i} : M_{i} \times M_{j} \to M_{i}$
is the projection.
By a completely similar discussion as in \cite[5.11.11]{CG97},
we can prove the following.
\begin{Prop}
\label{Prop:RR}
The $G$-equivariant Riemann-Roch homomorphisms 
are compatible with the convolution product, i.e.~we have
$$
\RR^{G}(\zeta * \eta) = \RR^{G}(\zeta) 
* \RR^{G}(\eta), \quad
\zeta \in \hK^{G}(Z_{12})_{\kk}, 
\eta \in \hK^{G}(Z_{23})_{\kk}.
$$ 
\end{Prop}

\subsection{Quiver Hecke algebra}
\label{Ssec:KLR}

Fix an element $\beta = \sum_{i \in I}
d_{i} \alpha_{i}
\in \cQ^{+}$ 
and put $d := \sum_{i \in I} d_{i}$.
Let
$$
I^{\beta} := \{ 
\ii = (i_{1}, \ldots, i_{d}) \in I^{d} \mid
\alpha_{i_{1}} + \cdots + \alpha_{i_{d}} = \beta 
\}.
$$
The symmetric group $\SG_{d}$
of degree $d$ acts on the set $I^{\beta}$ 
from the right by
$(i_{1}, \ldots, i_{d}) \cdot w := (i_{w(1)}, \ldots
, i_{w(d)}).$
Let $s_{k} \in \SG_{d}$ 
denote the transposition of $k$ and $k+1$
for $1 \le k < d$.

\begin{Def}[Khovanov-Lauda \cite{KL09}, Rouquier \cite{Rouquier08}]
\label{Def:KLR}
The quiver Hecke algebra $H_{Q}(\beta)$
is defined to be a $\kk$-algebra with the generating set
$
\{ e( \ii ) \mid \ii \in I^{\beta}\}
\cup 
\{x_{1}, \ldots, x_{d} \}
\cup
\{\tau_{1}, \ldots, \tau_{d-1} \},
$
satisfying the following relations:
$$
e(\ii) e(\ii^{\prime}) = 
\delta_{\ii, \ii^{\prime}} e(\ii), \quad
\sum_{\ii \in I^{\beta}} e(\ii) = 1,
\quad
x_{k} x_{l} = x_{l} x_{k}, \quad
x_k e(\ii) = e(\ii) x_{k},  
$$
$$
\tau_{k} e(\ii) = e(\ii \cdot s_k) \tau_k,
\quad
\tau_k \tau_l = \tau_l \tau_k \quad \text{if $|k-l| > 1$},
$$
$$
\tau_{k}^{2} e(\ii) = 
\begin{cases}
(x_{k} - x_{k+1})e(\ii), 
& \text{if $i_{k} \leftarrow i_{k+1}$},\\
(x_{k+1} - x_{k})e(\ii), 
& \text{if $i_{k} \to i_{k+1}$},\\ 
e(\ii) & \text{if $a_{i_{k}, i_{k+1}} = 0$},\\
0
& \text{if $i_{k}=i_{k+1}$},
\end{cases}
$$
$$
(\tau_k x_l - x_{s_{k}(l)} \tau_k)e(\ii)
= \begin{cases}
- e(\ii) & \text{if $l=k, i_{k} = i_{k+1}$}, \\
e(\ii) & \text{if $l=k+1, i_{k} = i_{k+1}$}, \\
0 & \text{otherwise},
\end{cases}  
$$
$$
(\tau_{k+1} \tau_{k} \tau_{k+1} - \tau_{k} \tau_{k+1} \tau_{k}) 
e(\ii)
= 
\begin{cases}
e(\ii)
& \text{if $i_{k} = i_{k+2}, i_{k} \leftarrow i_{k+1}$}, \\
-e(\ii) 
& \text{if $i_{k} = i_{k+2}, i_{k} \to i_{k+1}$}, \\
0 &
\text{otherwise}.
\end{cases}  
$$
\end{Def}

The quiver Hecke algebra $H_{Q}(\beta)$ is equipped with 
a $\Z$-grading given by
$$
\deg e(\ii) = 0, \quad
\deg x_k = 2, \quad
\deg \tau_{k} e(\ii) = -a_{i_k, i_{k+1}}.
$$
Since the grading is bounded from below
(see \cite[Theorem~2.5]{KL09}),
the completion $\widehat{H}_{Q}(\beta)
:= H_{Q}(\beta)^{\wedge}$
inherits a natural structure of $\kk$-algebra.

We recall the faithful polynomial
right representation of $H_{Q}(\beta)$
from \cite[Section 2.3]{KL09}.
We set 
$$
P_{\beta} := \bigoplus_{\ii \in I^{\beta}}
\kk [x_{1}, \ldots, x_{d}] 1_{\ii}
$$
with a commutative $\kk[x_{1}, \ldots, x_{d}]$-algebra 
structure $1_{\ii} \cdot 1_{\ii^{\prime}} = 
\delta_{\ii \ii^{\prime}} 1_{\ii}$.  
We define $f^{w}(x_{1}, \ldots, x_{d})
:= f(x_{w(1)}, \ldots, x_{w(d)})$ for
$f \in \kk[x_{1}, \ldots, x_{d}]$
and $w \in \SG_{d}$.

\begin{Thm}[\cite{KL09} Proposition 2.3]
\label{Thm:KL}
The following formulas give a faithful 
right $H_{Q}(\beta)$-module structure on 
the $\kk$-vector space $P_{\beta}$:
\begin{align*}
a \cdot e(\ii) &= a 1_{\ii}, \\
a \cdot x_{k} &= a x_{k}, \\
(f 1_{\ii}) \cdot \tau_{k} &=
\begin{cases}
\displaystyle 
\frac{f^{s_{k}} - f}{x_{k} - x_{k+1}} 1_{\ii} & \text{if $i_{k} = i_{k+1}$},\\
 (x_{k+1} - x_{k})f^{s_{k}} 1_{\ii \cdot s_{k}} & 
\text{if $i_{k} \leftarrow i_{k+1}$},\\
f^{s_{k}} 1_{\ii \cdot s_{k}} & \text{otherwise},
\end{cases}        
\end{align*} 
where $a \in P_{\beta}$ and 
$f 1_{\ii} \in \kk[x_{1}, \ldots, x_{d}]1_{\ii}$.
\end{Thm} 

Replacing the polynomial ring $\kk[x_{1}, \ldots, x_{d}]$
with the ring $\kk[\![ x_{1}, \ldots, x_{d} ]\!]$ of formal power series,
we get the completion of the representation $P_{\beta}$:
\begin{equation}
\label{Eq:Phat}
\widehat{P}_{\beta}
:= \bigoplus_{\ii \in I^{\beta}}
\kk[\![ x_{1}, \ldots, x_{d} ]\!] 1_{\ii}
= P_{\beta} \otimes_{H_{Q}(\beta)} \widehat{H}_{Q}(\beta).
\end{equation}


\subsection{Varagnolo-Vasserot\rq{}s realization}
\label{Ssec:VV}

Fix an $I$-graded $\C$-vector space $D=\bigoplus_{i \in I} D_{i}$
with $\vdim D = \beta$,
i.e. $\dim D_{i} = d_i$
as in Subsection \ref{Ssec:Rep_quiver}.
We consider the following two non-singular $G_{\beta}$-varieties:
\begin{align*}
\B_{\beta} &= 
\{F^{\bullet} = (D=F^{0} \supsetneq F^{1} \supsetneq \cdots 
\supsetneq F^{d}=0) \mid \text{$F^{k}$ is an $I$-graded 
subspace of $D$} \}, \\
\F_{\beta} &=
\{(F^{\bullet}, x) \in \B_{\beta} \times E_{\beta}
\mid x(F^{k}) \subset F^{k} \; \text{for any $1\le k \le d$} \}.
\end{align*}
The $G_\beta$-action on $\F_{\beta}$ is defined so that
the projections 
$\mathrm{pr}_{1} : \F_{\beta} \to \B_{\beta}$ and
$\mu_{\beta} := \mathrm{pr}_{2}: \F_{\beta} \to E_{\beta}$
are $G_{\beta}$-equivariant. 
They decompose into connected components as
$$
\B_{\beta} = \bigsqcup_{\ii \in I^{\beta}} \B_{\ii},
\qquad
\F_{\beta} = \bigsqcup_{\ii \in I^{\beta}} \F_{\ii}, 
$$
where we put 
$$
\B_{\ii} := \{ F^{\bullet} \in \B_{\beta} \mid 
\vdim F^{k-1} = \vdim F^{k} + \alpha_{i_k}, \; \forall k \},
\quad \F_{\ii} := (\mathrm{pr}_{1})^{-1}(\B_{\ii})
$$
for $\ii=(i_{1}, \ldots, i_{d}) \in I^{\beta}$.

We fix a basis $\{ v_{k} \}_{1 \le k\le d}$
of the vector space $D$ so that
the set $\{ v_{i,j}\}_{1 \le j \le d_{i}}$
forms a basis of the vector space $D_{i}$ for each $i \in I$, 
where we put
$
v_{i,j} := v_{d_{1} + \cdots + d_{i-1} + j}.
$
Let $H_{i} \subset GL(D_{i})$
be the maximal torus 
fixing the lines $\{\C v_{i,j}\}_{1\le j \le d_{i}}$
for each $i \in I$.
We set
$H_{\beta} := \prod_{i \in I} H_{i} \subset G_{\beta}$.

Let
$
F_{0}^{\bullet} \in \B_{\beta}
$
be the flag
defined by
$
F_{0}^{k} := \bigoplus_{l > k} \C v_{l},
$
which belongs to the component $\B_{\ii_0}$
with 
$\ii_0 := (1^{d_{1}}, 2^{d_{2}}, \ldots, n^{d_{n}}) \in I^{\beta}$.
For each $\ii \in I^{\beta}$, we fix an element 
$w_{\ii} \in \SG_{d}$ such that $\ii = \ii_{0} \cdot w_{\ii}$.
The set $\{ w_{\ii} \}_{\ii \in I^{\beta}}$ forms a complete
system of coset representatives for the quotient 
$\SG_{\beta} \backslash \SG_{d}$, where
$\SG_{\beta} := \Stab_{\SG_{d}}(\ii_0) 
= \SG_{d_1} \times \cdots \times \SG_{d_n}$.
For each $w \in \SG_{d}$, we define the flag $F^{\bullet}_{w}$ by
$
F_{w}^{k} := \bigoplus_{l > k} \C v_{w(l)}
$
which belongs to the component $\B_{\ii_{0} \cdot w}$.
Let
$F^{\bullet}_{\ii} := F^{\bullet}_{w_{\ii}} \in \B_{\ii}$
for $\ii \in I^{\beta}$.
Then we have
$\B_{\ii} \cong G_{\beta}/B_{\ii}$
with $B_{\ii} := \Stab_{G_{\beta}}(F_{\ii}^{\bullet})
\subset G_{\beta}$
being the Borel subgroup fixing the flag $F^{\bullet}_{\ii}$,
which
contains the maximal torus $H_{\beta}$.
Then
we have
\begin{equation}
\label{Eq:isomRH}
H^{G_{\beta}}_{*}(\B_{\ii}, \kk)
\cong H^{B_{\ii}}_{*}(\mathrm{pt}, \kk)
\cong H_{H_{\beta}}^{*}(\mathrm{pt}, \kk)
\cong \kk[x_{1}, \ldots, x_{d}] 1_{\ii},
\end{equation}
where the last isomorphism sends 
the $1$st $H_{\beta}$-equivariant 
Chern class of the line $\C v_{w_{\ii}(k)}$
to the element $x_{k}1_{\ii}$.
Thus we get an isomorphism
\begin{equation}
\label{Eq:isomP}
H_{*}^{G_{\beta}}(\B_{\beta}, \kk) 
= \bigoplus_{\ii \in I^{\beta}}
H_{*}^{G_{\beta}}(\B_{\ii}, \kk)
\cong \bigoplus_{\ii \in I^{\beta}} 
\kk[x_{1}, \ldots, x_{d}] 1_{\ii} 
= P_{\beta}.
\end{equation}

We consider the Steinberg type variety 
$
\cZ_\beta := \F_{\beta} \times_{E_{\beta}} \F_{\beta} 
$
associated with the morphism $\mu_{\beta}: \F_{\beta} \to E_{\beta}$.
Its $G_{\beta}$-equivariant 
Borel-Moore homology group
$
H^{G_{\beta}}_{*}(\cZ_{\beta}, \kk)$
becomes a $\kk$-algebra with respect to   
the convolution product relative to 
$\F_{\beta} \times \F_{\beta} \times \F_{\beta}$.
We identify the variety $\B_{\beta}$
with the fiber product
$\{ 0 \} \times_{E_{\beta}} \F_{\beta}$.
Then the convolution product
relative to $\{ 0 \} \times \F_{\beta} \times \F_{\beta}$
makes the space $H_{*}^{G_{\beta}}(\B_{\beta}, \kk)$
into a right $H_{*}^{G_{\beta}}(\cZ_{\beta}, \kk)$-module. 

Let $\mu_{\ii}$ denote the
restriction of the proper morphism $\mu_{\beta}: \F_{\beta} \to E_{\beta}$
to the component $\F_{\ii}$ for $\ii \in I^{\beta}$.
We put
$$
\mathcal{L}_{\beta}
:= \bigoplus_{\ii \in I^{\beta}} 
(\mu_{\ii})_{*} \underline{\kk} [\dim \F_{\ii}],
$$
where
$\underline{\kk}[\dim \F_{\ii}]$
is the trivial local system 
(i.e. the constant $\kk$-sheaf of rank $1$)
on $\F_{\ii}$
homologically shifted by $\dim \F_{\ii}$.
By the decomposition theorem, we have
$$
\mathcal{L}_{\beta}
\cong \bigoplus_{\mm \in \KP(\beta)}
L_{\mm} \otimes_{\kk} \mathcal{IC}_{\mm}
= \bigoplus_{\mm \in \KP(\beta)} 
\bigoplus_{k \in \Z} L_{\mm, k} \otimes_{\kk} \mathcal{IC}_{\mm}[k],
$$
where 
$\mathcal{IC}_{\mm}$ 
denotes the intersection cohomology complex associated 
with the trivial local system on the orbit $\mathbb{O}_{\mm}$
and $L_{\mm} = \bigoplus_{k \in \Z} L_{\mm, k}[k]$ is a
self-dual 
finite-dimensional graded $\kk$-vector space
for each $\mm \in \KP(\beta)$.
The vector space $L_{\mm}$ is known to be non-zero 
for all $\mm \in \KP(\beta)$
(see \cite[Corollary 2.8]{Kato14}).
We consider the Yoneda algebra
$$
\Ext_{G_{\beta}}^{*}
(\mathcal{L}_{\beta},
\mathcal{L}_{\beta})
= \bigoplus_{k \in \Z} \Ext_{G_{\beta}}^{k}(
\mathcal{L}_{\beta},
\mathcal{L}_{\beta})
$$
in the derived category of 
$G_{\beta}$-equivariant constructible complexes on $E_{\beta}$.
This is a $\Z$-graded $\kk$-algebra whose grading
is bounded from below.

By a standard argument (see \cite[Section 8.6]{CG97}),
we have an isomorphism of $\kk$-algebras
\begin{equation}
\label{Eq:isomG1}
\Ext_{G_{\beta}}^{*}(\mathcal{L}_{\beta},
\mathcal{L}_{\beta})
\cong H^{G_{\beta}}_{*}(\cZ_{\beta}, \kk).
\end{equation}
Note that this is not
compatible with the $\Z$-grading.

Let $\mathcal{L}_{\ii}(k)$ be
the $G_{\beta}$-equivariant line bundle
on $\F_{\ii}$ whose fiber at the point $(F^{\bullet}, x) \in \F_{\ii}$
is $F^{k-1} / F^{k}$
for $\ii \in I^{\beta}$ and $1 \le k \le d$.

\begin{Thm}[Varagnolo-Vasserot \cite{VV11}]
\label{Thm:VV}
There is a unique isomorphism of $\Z$-graded $\kk$-algebras
\begin{equation}
\label{Eq:isomVV}
H_{Q}(\beta) \xrightarrow{\cong} \Ext_{G_{\beta}}^{*}
(\mathcal{L}_{\beta}, \mathcal{L}_{\beta})
\end{equation}
which satisfies the following properties:
\begin{enumerate}
\item \label{Thm:VV:elem}
The composition $H_{Q}(\beta) 
\xrightarrow{\cong} H^{G_{\beta}}(\cZ_{\beta}, \kk)$ 
of the isomorphisms (\ref{Eq:isomVV})
and (\ref{Eq:isomG1}) sends the element $e(\ii)$ (resp.~$x_{k} e(\ii)$)
to the push-forward of the fundamental
class $[\F_{\ii}]$ (resp.~the $1$st $G_{\beta}$-equivariant Chern class 
of the line bundle $\mathcal{L}_{\ii}(k)$) with respect to the diagonal embedding
$\F_{\ii} \to \F_{\ii} \times_{E_{\beta}} \F_{\ii}$;
\item We have the following commutative diagram:
$$
\xy
\xymatrix{
H_{Q}(\beta)
\ar[r]^-{\cong}
\ar[d]
& 
H_{*}^{G_{\beta}}(\cZ_{\beta}, \kk) 
\ar[d]
\\
\End \left(P_{\beta} \right)^{\mathrm{op}}
\ar[r]^-{\cong}
&
\End \left(
H_{*}^{G_{\beta}}(\B_{\beta}, \kk) \right)^{\mathrm{op}},
}
\endxy
$$
where the lower horizontal arrow 
denotes the isomorphism induced from (\ref{Eq:isomP}) 
and the vertical arrows denote the right module structures.
\end{enumerate}
\end{Thm}

\begin{Rem}
Because our convention of the flag variety $\B_{\beta}$ 
differs from Varagnolo-Vasserot\rq{}s \cite{VV11},
we need a modification.
Actually, our isomorphism 
(\ref{Eq:isomVV}) is obtained by 
twisting the  
original isomorphism
$H_{Q}(\beta) \cong \Ext_{G_{\beta}}^{*}(
\mathcal{L}_{\beta}, \mathcal{L}_{\beta})$
in \cite{VV11} by a $\kk$-algebra involution
on $H_{Q}(\beta)$ given by
$$
e(\ii) \mapsto e(\ii^{\mathrm{op}}),  \quad
x_{k} \mapsto x_{d-k+1}, 
\quad
\tau_{k}e(\ii)  \mapsto 
\begin{cases}
-\tau_{d-k} e(\ii^{\mathrm{op}}) & \text{if $i_k = i_{k+1}$}; \\
\tau_{d-k} e(\ii^{\mathrm{op}}) & \text{if $i_k \neq i_{k+1}$},
\end{cases}
$$
where $\ii^{\mathrm{op}} := (i_{d}, \ldots, i_{2}, i_{1})$ 
for $\ii = (i_{1}, i_{2}, \ldots, i_{d}) \in I^{\beta}$.
\end{Rem}

Similarly to the case of
the $G_{\beta}$-equivariant Borel-Moore homologies,
the $K$-group $K^{G_{\beta}}(\cZ_{\beta})_{\kk}$
becomes an $R(G_{\beta})_{\kk}$-algebra 
and the $K$-group $K^{G_{\beta}}(\B_{\beta})_{\kk}$
becomes a right 
$K^{G_{\beta}}(\cZ_{\beta})_{\kk}$-module
with respect to the convolution products.

For each $\ii \in I^{\beta}$, we have
$$
K^{G_{\beta}}(\B_{\ii})_{\kk} \cong 
K^{B_{\ii}}(\mathrm{pt})_{\kk} \cong 
K^{H_{\beta}}(\mathrm{pt})_{\kk}
= R(H_{\beta})_{\kk}
\cong \kk[y^{\pm 1}_{1}, \ldots, y^{\pm 1}_{d}] 1_{\ii}
$$
where the last isomorphism sends
the class $[\C v_{w_{\ii}(k)}]$
of the $1$-dimensional
$H_{\beta}$-module $\C v_{w_{\ii}(k)}$
to the element $y_{k} 1_{\ii}$. 
The $G_{\beta}$-equivariant 
Chern character map
$
(\ch^{G_{\beta}})^{\B_{\ii}}_{\B_{\ii}}
$
gives an isomorphism of $\kk$-algebras
$$
\hK^{G_{\beta}}(\B_{\ii})_{\kk}
\cong \kk[\![ y_{1} -1, \ldots, y_{d}-1]\!] 1_{\ii}
\xrightarrow{\cong}
\kk[\![ x_{1}, \ldots, x_{d} ]\!] 1_{\ii}
\cong H_{*}^{G_{\beta}}(\B_{\ii}, \kk)^{\wedge}, 
$$
where the middle arrow 
sends the element $y_{k} 1_{\ii}$
to the exponential $e^{x_{k}} 1_{\ii}$
for $1 \le k \le d$.
Applying the equivariant Riemann-Roch theorem (=Theorem~\ref{Thm:EG})
to the inclusion $\B_{\ii} \hookrightarrow \F_{\ii}$,
we have
\begin{equation}
\label{Eq:const}
(\ch^{G_{\beta}})^{\F_{\ii}}_{\B_{\ii}} = C_{\ii} \cdot (\ch^{G_{\beta}})^{\B_{\ii}}_{\B_{\ii}},
\quad
C_{\ii} := (\Td^{G_{\beta}}_{\F_{\ii}})^{-1} \Td^{G_{\beta}}_{\B_{\ii}} \cdot 1_{\ii}
\in \kk[\![ x_{1}, \ldots, x_{d} ]\!] 1_{\ii}
\end{equation}
and hence the map
$(\ch^{G_{\beta}})^{\F_{\ii}}_{\B_{\ii}}$ 
is an isomorphism of $\hR(G_{\beta})_{\kk}$-modules.
Summing up over $\ii \in I^{\beta}$, we obtain
an isomorphism of $\hR(G_{\beta})_{\kk}$-modules
\begin{equation}
\label{Eq:isomRRB}
(\ch^{G_{\beta}})^{\F_{\beta}}_{\B_{\beta}} : 
\hK^{G_{\beta}}(\B_{\beta})_{\kk} \xrightarrow{\cong} 
H_{*}^{G_{\beta}}(\B_{\beta}, \kk)^{\wedge}. 
\end{equation}

\begin{Prop}
\label{Prop:RRflag}
The Riemann-Roch homomorphism
gives an isomorphism of $\widehat{R}(G_{\beta})_{\kk}$-algebras:
$$ \RR^{G_{\beta}}  :
\hK^{G_{\beta}}(\cZ_{\beta})_{\kk}
\xrightarrow{\cong} 
H_{*}^{G_{\beta}}(\cZ_{\beta}, \kk)^{\wedge},
$$
which makes the following diagram commute:
\begin{equation}
\label{Diag:sqKH}
\xy
\xymatrix{
\hK^{G_{\beta}}(\cZ_{\beta})_{\kk}
\ar[r]^-{\cong}
\ar[d]
& 
H_{*}^{G_{\beta}}(\cZ_{\beta}, \kk)^{\wedge} 
\ar[d]
\\
\End \left(\hK^{G_{\beta}}(\B_{\beta})_{\kk} \right)^{\mathrm{op}}
\ar[r]^-{\cong}
&
\End \left(
H_{*}^{G_{\beta}}(\B_{\beta}, \kk)^{\wedge} \right)^{\mathrm{op}},
}
\endxy
\end{equation}
where the lower horizontal arrow denotes
the isomorphism induced from (\ref{Eq:isomRRB}) and  
the vertical arrows denote the right module structures.
\end{Prop}

\begin{proof}
By Proposition~\ref{Prop:RR},
the map $\RR^{G_{\beta}} :
\hK^{G_{\beta}}(\cZ_{\beta})_{\kk}
\to H_{*}^{G_{\beta}}(\cZ_{\beta}, \kk)^{\wedge}$ is an
algebra homomorphism and the diagram 
(\ref{Diag:sqKH}) commutes.
To prove that the map 
$\RR^{G_{\beta}} :
\hK^{G_{\beta}}(\cZ_{\beta})_{\kk}
\to H_{*}^{G_{\beta}}(\cZ_{\beta}, \kk)^{\wedge}$ is an isomorphism,
it suffices to check that
the equivariant Chern character map $(\ch^{G_{\beta}}
)^{\F_{\beta} \times \F_{\beta}}_{\cZ_{\beta}} :
\hK^{G_{\beta}}(\cZ_{\beta})_{\kk}
\to H_{*}^{G_{\beta}}(\cZ_{\beta}, \kk)^{\wedge}$ 
gives an isomorphism of $\widehat{R}(G_{\beta})_{\kk}$-modules
since $\RR^{G_{\beta}}$ is obtained from 
$(\ch^{G_{\beta}}
)^{\F_{\beta} \times \F_{\beta}}_{\cZ_{\beta}}$
by multiplying
the $G_{\beta}$-equivariant Todd class
$p_{1}^{*}\Td^{G_{\beta}}_{\F_{\beta}}$, which  
is an invertible element. 
Because we have the connected component decomposition 
$$
\cZ_{\beta} = \bigsqcup_{\ii, \ii^{\prime} \in I^{\beta}}
\cZ_{\ii, \ii^{\prime}}, \quad
\cZ_{\ii, \ii^{\prime}} :=
\F_{\ii} \times_{E_{\beta}} \F_{\ii^{\prime}},
$$
we focus on a connected component 
$$
\cZ_{\ii, \ii^{\prime}}
= \{ (F^{\bullet}, F^{\prime \bullet}, x) \in
\B_{\ii} \times \B_{\ii^{\prime}} \times E_{\beta} \mid
x(F^{k}) \subset F^{k}, x(F^{\prime k}) \subset F^{\prime k}, \; 
\forall k \}.$$
For each $w \in \SG_{\beta} w_{\ii^{\prime}}$, we define 
a locally closed $G_{\beta}$-subvariety
$$
\cZ_{\ii, \ii^{\prime}}^{w}
= G_{\beta} \times^{B_{\ii}} \{ (F_{\ii}^{\bullet}, F^{\prime \bullet}, 
x ) \in \cZ_{\ii, \ii^{\prime}} \mid F^{\prime \bullet} \in 
B_{\ii} F_{w}^{\bullet} \}
$$
which is a $G_{\beta}$-equivariant affine bundle 
over $\B_{\ii}$.
They give a $G_{\beta}$-stable 
stratification
$
\cZ_{\ii, \ii^{\prime}} := 
\bigsqcup_{w \in \SG_{\beta} w_{\ii^{\prime}}}
Z_{\ii, \ii^{\prime}}^{w}.
$
Fix a total ordering 
$\SG_{\beta}w_{\ii^{\prime}}
= \{w_{1}, w_{2}, \ldots, w_{m} \}$
such that we have $w_{k}w_{\ii}^{-1} < w_{l} w_{\ii}^{-1}$ 
in the Bruhat ordering
only if $k<l$.
We simply write $\cZ_{\ii, \ii^{\prime}}^{k}
:= \cZ_{\ii, \ii^{\prime}}^{w_k}$ and set
$\cZ_{\ii, \ii^{\prime}}^{\le k} := 
\bigsqcup_{j \le k} \cZ_{\ii, \ii^{\prime}}^{j}$.
Then for each $k$,
the variety $\cZ_{\ii, \ii^{\prime}}^{\le k-1}$
is closed in $\cZ_{\ii, \ii^{\prime}}^{\le k}$
and its complement is $\cZ_{\ii, \ii^{\prime}}^{k}$.
Since $\cZ_{\ii, \ii^{\prime}}^{k}$ is 
a $G_{\beta}$-equivariant affine bundle over $\B_{\ii}$,
its homology of odd degree vanishes:
$H_{\mathrm{odd}}^{G_{\beta}}(\cZ_{\ii, \ii^{\prime}}^{k}, \kk)=0$.
Therefore an inductive argument with respect to $k$ 
yields
$H_{\mathrm{odd}}^{G_{\beta}}(\cZ_{\ii, \ii^{\prime}}^{\le k}, \kk) =0$.
Using the cellular fibration lemma \cite[5.5.1]{CG97} for 
equivariant $K$-groups and Proposition~\ref{Prop:chex}, 
we obtain 
the following commutative diagram with exact rows
for each $k$:
$$
\xy
\xymatrix{
0
\ar[r]
&
\hK^{G_{\beta}}(\cZ_{\ii, \ii^{\prime}}^{\le k-1})_{\kk}
\ar[r]
\ar[d]^-{\ch^{G_{\beta}}}
& 
\hK^{G_{\beta}}(\cZ_{\ii, \ii^{\prime}}^{\le k})_{\kk}
\ar[r]
\ar[d]^-{\ch^{G_{\beta}}}
&
\hK^{G_{\beta}}(\cZ_{\ii, \ii^{\prime}}^{k})_{\kk}
\ar[r]
\ar[d]^-{\ch^{G_{\beta}}}
&
0
\\
0
\ar[r]
&
H_{*}^{G_{\beta}}(\cZ_{\ii, \ii^{\prime}}^{\le k-1},
\kk)^{\wedge}
\ar[r]
&
H_{*}^{G_{\beta}}(\cZ_{\ii, \ii^{\prime}}^{\le k}, \kk)^{\wedge}
\ar[r]
&
H_{*}^{G_{\beta}}(\cZ_{\ii, \ii^{\prime}}^{k}, \kk)^{\wedge}
\ar[r]
&
0.
}
\endxy
$$ 
Note that the map
$\ch^{G_{\beta}} : 
\hK^{G_{\beta}}(\cZ^{k}_{\ii, \ii^{\prime}})_{\kk}
\to 
H_{*}^{G_{\beta}}(\cZ_{\ii, 
\ii^{\prime}}^{k}, \kk)^{\wedge}$
is an isomorphism for any $k$
since again the variety $Z^{k}_{\ii, \ii^{\prime}}$
is an affine bundle over $\B_{\ii}$.
Hence, by induction on $k$, we conclude that 
$\ch^{G_{\beta}} : 
\hK^{G_{\beta}}(\cZ^{\le k}_{\ii, \ii^{\prime}})_{\kk}
\to 
H_{*}^{G_{\beta}}(\cZ_{\ii, 
\ii^{\prime}}^{\le k}, \kk)^{\wedge}$
is an isomorphism for all $k$.
\end{proof}

Note that
the isomorphism (\ref{Eq:isomG1}) induces 
an isomorphism between the completions:
$$
\Ext_{G_{\beta}}^{*}(\mathcal{L}_{\beta}, 
\mathcal{L}_{\beta})^{\wedge}
\cong
H_{*}^{G_{\beta}}(\cZ_{\beta}, \kk)^{\wedge}. 
$$
As a summary of this subsection, we have the following.

\begin{Cor}
We have the following isomorphisms of $\kk$-algebras:
$$
\widehat{H}_{Q}(\beta) 
\cong \Ext_{G_{\beta}}^{*}(\mathcal{L}_{\beta},
\mathcal{L}_{\beta})^{\wedge} 
\cong
H_{*}^{G_{\beta}}(\cZ_{\beta}, \kk)^{\wedge}
\cong
\hK^{G_{\beta}}(\cZ_{\beta})_{\kk}.
$$
\end{Cor}


\subsection{Nakajima\rq{}s homomorphism and the category 
$\Cc_{Q, \beta}$}
\label{Ssec:Nakajima}

Henceforth, we specialize $\kk$ to be the field $\Q(q)$
of rational functions in an indeterminate $q$.
In this subsection,
we consider  
the quantum loop algebra $U_{q} \equiv U_{q}(L \g)$
defined over $\kk$.
The quantum loop algebra $U_{q}(L \g)$
is isomorphic to the
level zero quotient of the quantum affine algebra
$U_{q}^{\prime}(\widehat{\g})$
without the degree operator.
We do not recall the definitions here.
See e.g.~\cite{Fujita17}, \cite{KKK15}, \cite{Nakajima01}
for the precise definitions of $U_{q}(L\g)$ or 
$U^{\prime}_{q}(\widehat{\g})$.

Recall 
the quiver varieties with proper
$\G(\lambda)$-equivariant morphism
$\pi : \M(\lambda) \to \M_{0}(\lambda)$
for each $\lambda \in \cP^{+}$ 
(see Subsection \ref{Ssec:Quiver_var}).
We consider the Steinberg type variety
$Z(\lambda) := 
\M(\lambda) \times_{\M_{0}(\lambda)} \M(\lambda)$.
Then its $\G(\lambda)$-equivariant $K$-group
$K^{\G(\lambda)}(Z(\lambda))$ becomes
an $R(\G(\lambda))$-algebra with respect to the
convolution product relative to 
$\M(\lambda) \times \M(\lambda) \times \M(\lambda)$.
We identify the fiber $\LL(\lambda) = \pi^{-1}(0)$ with 
the fiber product $\M(\lambda) \times_{\M_{0}(\lambda)} \{0\}.$
Then the convolution product relative to 
$\M(\lambda) \times \M(\lambda) \times \{ 0 \}$
makes the $K$-group $K^{\G(\lambda)}(\LL(\lambda))$
into a left $K^{\G(\lambda)}(Z(\lambda))$-module.

Recall that $\G(\lambda) = G(\lambda) \times \C^{\times}$.
We set $A := R(\C^{\times})$
and identify $A = \Z[v^{\pm 1}]$ in the standard way.
Specializing $v$ to $q$,
we regard $\kk$ as an $A$-algebra.

\begin{Thm}[Nakajima \cite{Nakajima01} Theorem 9.4.1]
\label{Thm:Nakajima}
There exists a $\kk$-algebra homomorphism
$$
\Phi_{\lambda} : U_{q}(L\g) \to 
K^{\G(\lambda)}(Z(\lambda))\otimes_{A} \kk
$$
such that the pull-back
$$
\bW(\lambda) := \Phi_{\lambda}^{*} \left(
K^{\G(\lambda)}(\LL(\lambda))\otimes_{A}\kk
\right)
$$
is a cyclic $U_{q}(L\g)$-module
generated by an extremal weight vector
$w_{\lambda} := [\mathcal{O}_{\LL(0, \lambda)}] \in 
K^{\G(\lambda)}(\LL(0,\lambda))\otimes_{A} \kk$
of weight $\lambda$. Moreover the module $\bW(\lambda)$
is free of finite rank over  
$\End_{U_{q}}(\bW(\lambda)) \cong 
R(\G(\lambda)) \otimes_{A} \kk$.
\end{Thm}        

\begin{Rem}
The module $\bW(\lambda)$
is known to be isomorphic to the global Weyl module
defined by Chari-Pressley \cite{CP01}
and also to the level $0$ extremal weight module
defined by Kashiwara \cite{Kashiwara94}.
In particular, if $\lambda = \varpi_{i}$ for some $i \in I$,
the module $\bW(\varpi_{i})$ is isomorphic
to the affinization of the fundamental module $W(\varpi_{i})$
(see \cite{Kashiwara02}).
\end{Rem}

Take an element $\hlam \in \lP^{+}$
with $\cl (\hlam) = \lambda$ and 
recall the $1$-dimensional subtorus
$T(\hlam) \subset \G(\hlam) \subset \G(\lambda)$. 
We identify $R(T(\hlam)) = A$
via the isomorphism
$\prod_{i \in I}f_{i} \times \mathrm{id} : \C^{\times} \xrightarrow{\cong}
T(\hlam)$.
Let 
$\mathfrak{m}_{\hlam}$ be 
the kernel of the restriction
$
R(\G(\lambda))\otimes_{A} \kk
\to 
R(T(\hlam))\otimes_{A} \kk = \kk.
$
The corresponding specialization 
$
\bW(\lambda) / \mathfrak{m}_{\hlam}
\bW(\lambda)
$
(known as the local Weyl module defined in \cite{CP01})
has a unique simple quotient $L(\hlam)$
in $U_{q} \modfd$.

\begin{Def}[Hernandez-Leclerc \cite{HL15}]
We define the category $\Cc_{Q}$ (resp.~$\Cc_{Q, \beta}$
for each $\beta \in \cQ^{+}$)
to be the minimal Serre full subcategory of
the category $U_{q} \modfd $ of finite-dimensional $U_{q}(L\g)$-modules 
containing
the simple objects $\{ L(\hlam) \mid \hlam \in \lP^{+}_{0} \}$
(resp.~$\{ L(\mm) \mid \mm \in \KP(\beta) \}$),
where $\lP^{+}_{0} = \bigsqcup_{\beta \in \cQ^{+}}
\KP(\beta) \subset \lP^{+}$ is as in Subsection \ref{Ssec:HL}.  
\end{Def}

\begin{Rem}
\label{Rem:Grotisom}
Let $G$ be a linear algebraic group whose Lie algebra is $\g$
and $N$ be  the maximal
unipotent subgroup of $G$ corresponding to 
the positive roots.
Hernandez-Leclerc \cite{HL15} proved that
the category $\Cc_{Q}$ is a monoidal subcategory
and there is an isomorphism from the
complexified Grothendieck ring
$K(\Cc_{Q})_{\C}$ to
the coordinate ring $\C[N]$,
which sends the classes of simple objects
to the elements of the dual canonical basis bijectively.
Actually, Hernandez-Leclerc
established an isomorphism between their quantizations.
We have a block decomposition
$\Cc_{Q} = \bigoplus_{\beta \in \cQ^{+}} \Cc_{Q, \beta}$
satisfying $\Cc_{Q, \beta} \otimes \Cc_{Q, \beta^{\prime}}
\subset
\Cc_{Q, \beta + \beta^{\prime}}$
(see \cite[Section 2.6]{Fujita17}).
This decomposition corresponds
to the weight decomposition 
$\C[N] = \bigoplus_{\beta \in \cQ^{+}}\C[N]_{\beta}$. 
The isomorphism
$
\M_{0, \beta}^{\bullet}
\cong
E_{\beta}$ in Theorem \ref{Thm:HL}
was originally established in order to give a geometric interpretation 
to the isomorphism $K(\Cc_{Q, \beta})_{\C} \cong \C[N]_{\beta}$.
\end{Rem}

Now we fix an element 
$\beta \in \cQ^{+}$. 
In Subsection \ref{Ssec:HL},
we defined the graded quiver variety $\Mg_{\beta}$
with a canonical $\G_{\beta}$-equivariant proper morphism
$\pi_{\beta}: \Mg_{\beta} \to E_{\beta}$,
which is obtained from $\pi : \M(\lambda) \to 
\M_{0}(\lambda)$ with $\lambda = \cl(\hlam_{\beta})$
by taking the fixed locus 
with respect to the action of the $1$-dimensional torus
$T_{\beta} \subset \G_{\beta} \subset \G(\lambda)$.
We form the Steinberg type variety
$\Zg_{\beta} := \Mg_{\beta} \times_{E_{\beta}} \Mg_{\beta} 
= Z(\lambda)^{T_{\beta}}$.
Let $\rr_{\beta}$ be 
the kernel of the restriction
$R(\G_{\beta})\otimes_{A} \kk
\to R(T_{\beta}) \otimes_{A}\kk = \kk$.
Note that the decomposition (\ref{Eq:group}) 
$\G_{\beta} \cong G_{\beta} \times T_{\beta}$ 
yields an isomorphism 
$$
K^{\G_{\beta}}(X)\otimes_{A} \kk 
\cong K^{G_{\beta}}(X)_{\kk}
$$
for any $\G_{\beta}$-variety $X$ with
a trivial $T_\beta$-action.
In particular, we have an isomorphism 
$R(\G_{\beta})\otimes_{A} \kk
\cong R(G_{\beta})_{\kk}$ of $\kk$-algebras,
via which the maximal ideal $\rr_{\beta}
\subset R(\G_{\beta})\otimes_{A} \kk$
corresponds to the augmentation ideal $I \subset 
R(G_{\beta})_{\kk}$.  
Therefore we have an isomorphism 
\begin{equation}
\label{Eq:complG}
\left[
K^{\G_{\beta}}(X)\otimes_{A} \kk 
\right]_{\rr_{\beta}}^{\wedge}
\cong 
\hK^{G_{\beta}}(X)_{\kk},
\end{equation}
where $[ - ]_{\rr_{\beta}}^{\wedge}$
denotes the $\rr_{\beta}$-adic completion. 
We define the $\kk$-algebra homomorphism
$\widehat{\Phi}_{\beta} : U_{q}(L\g) \to \hK^{G_{\beta}}
(Z^{\bullet}_{\beta})_{\kk}$ as the following composition:
\begin{align*}
U_{q}(L\g) 
& \xrightarrow{\Phi_{\lambda}}
K^{\G(\lambda)}(Z(\lambda)) \otimes_{A} \kk
\\
&\to
K^{\G_{\beta}}(Z(\lambda)) \otimes_{A} \kk
&& \text{(restriction to $\G_{\beta} \subset \G(\lambda)$)}
\\
& \to
\left[
K^{\G_{\beta}}(Z(\lambda))\otimes_{A} \kk 
\right]_{\rr_{\beta}}^{\wedge}
&&
\text{($\rr_{\beta}$-adic completion)}
\\
&\cong
\left[
K^{\G_{\beta}}(\Zg_{\beta})\otimes_{A} \kk 
\right]_{\rr_{\beta}}^{\wedge}
&&
\text{(localization theorem)}
\\
&\cong 
\hK^{G_{\beta}}(\Zg_{\beta})_{\kk}.
&& \text{(isomorphism (\ref{Eq:complG}))}
\end{align*}

\begin{Thm}[\cite{Fujita17} Theorem~4.9]
\label{Thm:mine}
The pull-back 
along the homomorphism 
$\widehat{\Phi}_{\beta} : U_{q}(L \g) \to \hK^{G_{\beta}}
(\Zg_{\beta})_{\kk}$ induces an equivalence
$$\widehat{\Phi}_{\beta}^{*}:
\hK^{G_{\beta}}(\Zg_{\beta})_{\kk} \modfd
\xrightarrow{\simeq}
\Cc_{Q, \beta}$$
between the category 
$\hK^{G_{\beta}}(\Zg_{\beta})_{\kk} \modfd$
of finite-dimensional 
$\hK^{G_{\beta}}(\Zg_{\beta})_{\kk}$-modules
and the category $\Cc_{Q, \beta} \subset U_{q}\modfd$.
\end{Thm}

The next proposition is
a counterpart of Proposition~\ref{Prop:RRflag}.

\begin{Prop}
\label{Prop:RRquiver}
The Riemann-Roch homomorphism gives an
isomorphism of $\widehat{R}(G_{\beta})_{\kk}$-algebras:
$$
\RR^{G_{\beta}} :
\hK^{G_{\beta}}(\Zg_{\beta})_{\kk}
\xrightarrow{\cong}
H_{*}^{G_{\beta}}(\Zg_{\beta}, \kk)^{\wedge}.
$$
\end{Prop} 
\begin{proof}
As in the proof of Proposition~\ref{Prop:RRflag}, 
it suffices to prove that
the equivariant Chern character map 
$(\ch^{G_{\beta}})^{\Mg_{\beta} 
\times \Mg_{\beta}}_{Z^{\bullet}_{\beta}}:
\hK^{G_{\beta}}(\Zg_{\beta})_{\kk}
\to
H_{*}^{G_{\beta}}(\Zg_{\beta}, \kk)^{\wedge}$
is an isomorphism.

Note that the $G_{\beta}$-orbit stratification (\ref{Eq:decE}) 
yields a stratification of $\Zg_{\beta}$:
$$
Z^{\bullet}_{\beta} = \bigsqcup_{\mm \in \KP(\beta)}
Z^{\bullet}_{\beta} |_{\mathbb{O}_{\mm}},
\quad 
Z^{\bullet}_{\beta} |_{\mathbb{O}_{\mm}}
\cong G_{\beta} \times^{\Stab_{G_{\beta}}(x_{\mm})} 
\left( \pi_{\beta}^{-1}(x_{\mm}) \times \pi_{\beta}^{-1}(x_{\mm}) 
\right).
$$ 
Fix a total ordering
$\KP(\beta) 
= 
\{\mm_{1}, \mm_{2}, \ldots, \mm_{s}\}$
such that
we have $\mathbb{O}_{k} \subset \overline{\mathbb{O}}_{l}$
only if $k < l$.
Set 
$Z_{\beta}^{k} := \Zg_{\beta}|_{\mathbb{O}_{\mm_k}}$
and 
$Z_{\beta}^{\le k} := \bigsqcup_{j \le k} Z_{\beta}^{j}$.
Then
the variety $Z_{\beta}^{\le k-1}$ is a closed subvariety
of $Z_{\beta}^{\le k}$
whose complement is $Z_{\beta}^{k}$. 
 By Proposition~\ref{Prop:fiber} and the reduction, we have 
\begin{align*}
K^{G_{\beta}}(Z^{k}_{\beta})
&\cong 
K^{G(\mm_k)}(\Lg(\mm_k) \times \Lg(\mm_k)), \\
H_{*}^{G_{\beta}}(Z^{k}_{\beta}, \kk)
& \cong
H_{*}^{G(\mm_k)}(\Lg(\mm_k) \times \Lg(\mm_k), \kk)
\end{align*}
for each $k$.
Then, using \cite[Theorem~7.4.1]{Nakajima01},
we can prove that the equivariant Chern character map
gives an isomorphism
$\ch^{G_{\beta}} : 
\hK^{G_{\beta}} (Z_{\beta}^{k})_{\kk}
\xrightarrow{\cong}
H_{*}^{G_{\beta}}(Z_{\beta}^{k}, \kk)^{\wedge}
$
for each $k$. Moreover, we obtain
the following commutative diagram with exact rows
for each $k$:
$$
\xy
\xymatrix{
0
\ar[r]
&
\hK^{G_{\beta}}(Z_{\beta}^{\le k-1})_{\kk}
\ar[r]
\ar[d]^-{\ch^{G_{\beta}}}
& 
\hK^{G_{\beta}}(Z_{\beta}^{\le k})_{\kk}
\ar[r]
\ar[d]^-{\ch^{G_{\beta}}}
&
\hK^{G_{\beta}}(Z_{\beta}^{k})_{\kk}
\ar[r]
\ar[d]^-{\ch^{G_{\beta}}}
&
0
\\
0
\ar[r]
&
H_{*}^{G_{\beta}}(Z_{\beta}^{\le k-1},
\kk)^{\wedge}
\ar[r]
&
H_{*}^{G_{\beta}}(Z_{\beta}^{\le k}, \kk)^{\wedge}
\ar[r]
&
H_{*}^{G_{\beta}}(Z_{\beta}^{k}, \kk)^{\wedge}
\ar[r]
&
0.
}
\endxy
$$ 
By induction on $k$,  
the equivariant Chern character map
gives an isomorphism 
$\ch^{G_{\beta}} :
\hK^{G_{\beta}}(Z_{\beta}^{\le k})_{\kk}
\xrightarrow{\cong}
H_{*}^{G_{\beta}}(Z_{\beta}^{\le k}, \kk)^{\wedge}$
for all $k$.
\end{proof}

We consider the proper push-forward 
$$
\mathcal{L}_{\beta}^{\bullet} := (\pi_{\beta})_{*} \underline{\kk}
$$
of the 
trivial local system $ \underline{\kk}$ on $\Mg_\beta$.
By the decomposition theorem,
we have
$$
\mathcal{L}_{\beta}^{\bullet} \cong \bigoplus_{\mm \in \KP(\beta)}
L_{\mm}^{\bullet} \otimes_{\kk} \mathcal{IC}_{\mm}
=
\bigoplus_{\mm \in \KP(\beta)} \bigoplus_{k \in \Z}
L^{\bullet}_{\mm, k} \otimes_{\kk} \mathcal{IC}_{\mm}[k],
$$
where $L_{\mm}^{\bullet} = \bigoplus_{k} L^{\bullet}_{\mm, k}$
is a finite-dimensional 
graded $\kk$-vector space, which is 
known to be non-zero for each $\mm$
(see \cite[Theorem~14.3.2]{Nakajima01}). 
Similarly to the previous subsection,
we have a standard isomorphism of $\kk$-algebras
\begin{equation}
\label{Eq:isomG2}
\Ext_{G_{\beta}}^{*}(\mathcal{L}_{\beta}^{\bullet}, 
\mathcal{L}_{\beta}^{\bullet}) \cong 
H_{*}^{G_{\beta}}(\Zg_{\beta}, \kk),
\end{equation}
which also induces an isomorphism between completions.
\begin{Cor}
We have the following isomorphisms of $\kk$-algebras:
$$
\Ext_{G_{\beta}}^{*}(\mathcal{L}_{\beta}^{\bullet}, 
\mathcal{L}_{\beta}^{\bullet})^{\wedge}
\cong
H_{*}^{G_{\beta}}(\Zg_{\beta}, \kk)^{\wedge}
\cong 
\hK^{G_{\beta}}(\Zg_{\beta})_{\kk}.
$$
\end{Cor}


\section{Dynkin quiver type quantum affine Schur-Weyl duality}
\label{Sec:SW}

\subsection{Geometric construction of a bimodule
and a Morita equivalence}

We keep the notation in the previous sections.
In particular, $\kk = \Q(q).$
We fix an element $\beta = \sum_{i \in I}
d_{i} \alpha_{i}
\in \cQ^{+}$ and
put
$\lambda := \cl(\hlam_{\beta})
\in \cP^{+}$.
From the two $G_{\beta}$-equivariant 
proper morphisms
$
\pi_{\beta} : \Mg_{\beta} \to E_{\beta}
$
and 
$
\mu_{\beta} : \F_{\beta} \to E_{\beta},
$
we form the fiber product
$
\Mg_{\beta} \times_{E_{\beta}} \F_{\beta}
$.
The convolution products make its 
completed $G_{\beta}$-equivariant
$K$-group 
$
\hK^{G_{\beta}}(\Mg_{\beta} 
\times_{E_{\beta}} \F_{\beta})_{\kk}$
into a $(\hK^{G_{\beta}}(\Zg_{\beta})_{\kk},
\hK^{G_{\beta}}(\cZ_{\beta})_{\kk})$-bimodule.
More precisely, the convolution products give 
$\kk$-algebra homomorphisms 
$$
\hK^{G_{\beta}}(\Zg_{\beta})_{\kk}
\to
\End \left(\hK^{G_{\beta}}(\Mg_{\beta} 
\times_{E_{\beta}} \F_{\beta})_{\kk}\right)
\leftarrow
\hK^{G_{\beta}}(\cZ_{\beta})_{\kk}^{\mathrm{op}},
$$
whose images commute with each other.
In the rest of this subsection, we prove that 
this bimodule induces a Morita equivalence.

For a moment, 
we focus on a component 
$\Mg_{\beta} \times_{E_{\beta}} \F_{\ii}$
for a fixed $\ii \in I^{\beta}$.
Using the isomorphism
$\B_{\ii} 
\cong G_{\beta}/B_{\ii}$
with $B_{\ii} = \Stab_{G_{\beta}}(F^{\bullet}_{\ii})$,
we have
\begin{align}
\Mg_{\beta} \times_{E_{\beta}} \F_{\ii}
&\cong \Mg_{\beta} \times_{E_{\beta}} 
\left( G_{\beta} \times^{B_{\ii}} \mathrm{pr}_{1}^{-1}(F_{\ii}^{\bullet})
\right) \nonumber \\
&\cong
G_{\beta} \times^{B_{\ii}} \left(
\Mg_{\beta} \times_{E_{\beta}} 
\mathrm{pr}_{1}^{-1}(F_{\ii}^{\bullet})
\right), \label{Eq:isomatr}
\end{align}
where $\mathrm{pr}_{1}$ is the projection 
$\F_{\ii} \ni (F^{\bullet}, x)
\mapsto F^{\bullet} \in \B_{\ii}$.
We define a $1$-parameter subgroup
$
\rho_{\ii} : \C^{\times} \to H_{\beta} 
$
by $\rho_{\ii}(t) v_{w_{\ii}(k)} := t^{k} v_{w_{\ii}(k)}$
for $t \in \C^{\times}$.
Note that this depends on the choice of $w_{\ii} \in \SG_{d}$
fixed in Subsection \ref{Ssec:VV}.
We observe that  
$$
\mathrm{pr}_{1}^{-1}(F_{\ii}^{\bullet})
\cong
\{ x \in E_{\beta} \mid x(F_{\ii}^{k}) \subset F_{\ii}^{k}, \, \forall k \}
 =
\left\{ x \in E_{\beta} \; \middle| \; 
\lim_{t \to 0} \rho_{\ii}(t) x = 0 \right\}.
$$
Therefore we get
$$
\Mg_{\beta} \times_{E_{\beta}} 
\mathrm{pr}_{1}^{-1}(F_{\ii}^{\bullet})
\cong \left\{ x \in \Mg_{\beta}\; \middle| \;
\lim_{t \to 0} \rho_{\ii}(t) \pi_{\beta}(x) = 0 \right\}.
$$
Since the morphism
$\pi_{\beta} : \Mg_{\beta} \to E_{\beta}$
is the $T_{\beta}$-fixed part
of $\pi : \M(\lambda) \to \M_{0}(\lambda)$,
it is natural to consider the following 
subvariety of $\M(\lambda)$:
$$
\widetilde{\mathfrak{Z}}(\lambda ; w_{\ii})
:= \left\{ x \in \M(\lambda) \; \middle| \; \lim_{t \to 0} \rho_{\ii}(t) \pi(x)
= 0 \in \M_{0}(\lambda) \right\},
$$
which turns out to be 
the tensor product variety introduced 
by Nakajima \cite{Nakajima01t}.
Since the subgroups $T_{\beta}$ and 
$\rho_{\ii}(\C^{\times})$
commute with each other, 
we have
\begin{equation}
\label{Eq:obs2}
\Mg_{\beta} \times_{E_{\beta}} 
\mathrm{pr}_{1}^{-1}(F_{\ii}^{\bullet})
\cong 
\widetilde{\mathfrak{Z}}(\lambda ; w_{\ii})^{T_{\beta}}.
\end{equation}
Using (\ref{Eq:isomatr}), (\ref{Eq:obs2}) and the reduction,
we obtain
\begin{align}
K^{G_{\beta}}(\Mg_{\beta} \times_{E_{\beta}} \F_{\ii})
&\cong 
K^{H_{\beta}}(\tZ
(\lambda ; w_{\ii})^{T_{\beta}}), 
\label{Eq:redK} \\
H_{*}^{G_{\beta}}(\Mg_{\beta} \times_{E_{\beta}} \F_{\ii}, 
\kk)
&\cong 
H_{*}^{H_{\beta}}(\tZ(\lambda; w_{\ii})^{T_{\beta}}, 
\kk). \label{Eq:redH}
\end{align}

\begin{Prop}
\label{Prop:chbimodule}
The $G_{\beta}$-equivariant Chern character map gives 
an isomorphism:
$$
\ch^{G_{\beta}} : 
\hK^{G_{\beta}}(\Mg_{\beta} \times_{E_{\beta}} \F_{\ii})_{\kk}
\xrightarrow{\cong}
H_{*}^{G_{\beta}}(\Mg_{\beta} \times_{E_{\beta}} \F_{\ii}, 
\kk)^{\wedge}.
$$
\end{Prop}
\begin{proof}
Thanks to (\ref{Eq:redK}) and (\ref{Eq:redH}),
it is enough to show that the $H_{\beta}$-equivariant
Chern character map 
$$
\ch^{H_{\beta}}:
\hK^{H_{\beta}}(\tZ
(\lambda ; w_{\ii})^{T_{\beta}})_{\kk}
\to
H_{*}^{H_{\beta}}(\tZ(\lambda, w_{\ii})^{T_{\beta}}, 
\kk)^{\wedge}
$$ 
is an isomorphism.
This latter assertion follows from 
a $T_{\beta}$-fixed part analogue of
\cite[Theorem~3.10.~(1)]{Nakajima01t}.
\end{proof}

The $G_{\beta}$-equivariant Borel-Moore homology
$H_{*}^{G_{\beta}}(\Mg_{\beta} \times_{E_{\beta}} \F_{\beta}, \kk)$
becomes a 
$(H_{*}^{G_{\beta}}(\Zg_{\beta}, \kk),
H_{*}^{G_{\beta}}(\cZ_{\beta}, \kk))$-bimodule
by the convolution products, similarly to the case of $K$-groups.
On the other hand, 
the $\Ext$-group 
$\Ext_{G_{\beta}}^{*}(\mathcal{L}_{\beta}^{\bullet},
\mathcal{L}_{\beta})$
becomes a 
$(
\Ext_{G_{\beta}}^{*}(\mathcal{L}_{\beta}^{\bullet}, 
\mathcal{L}_{\beta}^{\bullet}),
\Ext_{G_{\beta}}^{*}(\mathcal{L}_{\beta}, 
\mathcal{L}_{\beta}
))$-bimodule by the Yoneda products. 
This bimodule 
$\Ext_{G_{\beta}}^{*}(\mathcal{L}_{\beta}^{\bullet},
\mathcal{L}_{\beta})$
gives a Morita equivalence 
between  
$\Ext_{G_{\beta}}^{*}(\mathcal{L}_{\beta}^{\bullet}, 
\mathcal{L}_{\beta}^{\bullet})$
and $
\Ext_{G_{\beta}}^{*}(\mathcal{L}_{\beta}, 
\mathcal{L}_{\beta})$
because $\mathcal{IC}_{\mm}$ 
appears as a non-zero direct summand
of both $\mathcal{L}_{\beta}$ and $\mathcal{L}^{\bullet}_{\beta}$
for each $\mm \in \KP(\beta)$.
Moreover, we have a standard isomorphism
\begin{equation}
\label{Eq:isomG3}
H_{*}^{G_{\beta}}(\Mg_{\beta} \times_{E_{\beta}} \F_{\beta}, \kk)
\cong
\Ext_{G_{\beta}}^{*}(\mathcal{L}_{\beta}^{\bullet}, \mathcal{L}_{\beta})
\end{equation}

\begin{Thm}
\label{Thm:Morita}
We have the following commutative diagram:
$$
\xy
\xymatrix{
\hK^{G_{\beta}}(\Zg_{\beta})_{\kk}
\ar[r]
\ar[d]^-{\cong}_{\RR^{G_{\beta}}}
&
\End \left(
\hK^{G_{\beta}}(\Mg_{\beta} \times_{E_{\beta}} \F_{\beta})_{\kk}
\right)
\ar[d]^-{\cong}_{\RR^{G_{\beta}}}
&
\hK^{G_{\beta}}(\cZ_{\beta})_{\kk}^{\mathrm{op}}
\ar[l]
\ar[d]^-{\cong}_{\RR^{G_{\beta}}}
\\
H_{*}^{G_{\beta}}(\Zg_{\beta}, \kk)^{\wedge} 
\ar[r]
\ar[d]^-{\cong}_-{(\ref{Eq:isomG2})}
&
\End \left( H_{*}^{G_{\beta}}(\Mg_{\beta} \times_{E_{\beta}} 
\F_{\beta},\kk)^{\wedge} \right)
\ar[d]^-{\cong}_-{(\ref{Eq:isomG3})}
&
H_{*}^{G_{\beta}}(\cZ_{\beta}, \kk)^{\wedge \mathrm{op}}
\ar[d]^-{\cong}_-{(\ref{Eq:isomG1})}
\ar[l]
\\
\Ext_{G_{\beta}}^{*}(\mathcal{L}_{\beta}^{\bullet},
\mathcal{L}_{\beta}^{\bullet})^{\wedge}
\ar[r]
&
\End \left(
\Ext_{G_{\beta}}^{*}(\mathcal{L}_{\beta}^{\bullet},
\mathcal{L}_{\beta})^{\wedge} \right)
&
\Ext_{G_{\beta}}^{*}(\mathcal{L}_{\beta},
\mathcal{L}_{\beta})^{\wedge \mathrm{op}},
\ar[l]
}
\endxy
$$
where each row denotes the bimodule structure 
defined above.
In particular, 
the bimodule $\hK^{G_{\beta}}(\Mg_{\beta} \times_{E_{\beta}}
\F_{\beta} )_{\kk}$ gives a Morita equivalence
between two convolution algebras 
$\hK^{G_{\beta}}(\Zg_{\beta})_{\kk}$
and $\hK^{G_{\beta}}(\cZ_{\beta})_{\kk}$.
\end{Thm}
\begin{proof}
The commutativity of the upper half (resp.~lower half) 
of the diagram follows from Proposition~\ref{Prop:RR}
(resp.~an equivariant version of \cite[Theorem~8.6.7]{CG97}).
\end{proof}

\subsection{The left action of $U_{q}(L\g)$}
 
In this subsection, 
we fix $\ii = (i_1, \ldots, i_d) \in I^{\beta}$
and investigate the $U_{q}(L\g)$-module
structure of 
the pull-back
$\widehat{\Phi}_{\beta}^{*} (
\hK^{G_{\beta}}(\Mg_{\beta} \times_{E_{\beta}}
\F_{\ii})_{\kk})$.

We use the following notation.
For each $i \in I$, we define
$\lambda_{i} := \cl(\phi(\alpha_i)) = \varpi_{j}$ and
$a_{i} := q^{p}$ if $\phi(\alpha_{i}) = (j, p) \in \widehat{I}$.
Recall from Theorem~\ref{Thm:Nakajima} that we have
\begin{equation}
\label{Eq:isomRfund}
\End_{U_{q}} ( \bW(\lambda_{i}))
\cong R(\G(\lambda_{i}))
\otimes_{A}\kk = R(G(\lambda_{i}))_{\kk} \cong
\kk[z^{\pm 1}_{\lambda_{i}}],
\end{equation}
where $z_{\lambda_i}$ denotes the class of the $1$-dimensional 
representation of $G(\lambda_{i}) = \C^{\times}$
of weight $1$.

We recall some properties
of the tensor product variety
$\tZ(\lambda ; w_{\ii})$.
Let 
$$
\HH_{\beta} := H_{\beta} \times \C^{\times}
\subset
G_{\beta} \times \C^{\times}
= \G_{\beta} 
\subset \G(\lambda)
$$
be a maximal torus.
By construction,
the subvariety $\tZ(\lambda ; w_{\ii})
\subset \M(\lambda)$
is stable under the action of $\HH_{\beta}$. 
The convolution product makes 
the $\HH_{\beta}$-equivariant $K$-group
$K^{\HH_{\beta}}(\widetilde{\mathfrak{Z}}(\lambda; w_{\ii}))$
into a left $K^{\HH_{\beta}}(Z(\lambda))$-module.
Via the composition of the homomorphisms
$$
U_{q}(L\g) 
\xrightarrow{\Phi_{\lambda}} 
K^{\G(\lambda)}(Z(\lambda))\otimes_{A} \kk
\to
K^{\HH_{\beta}}(Z(\lambda)) \otimes_{A} \kk,  
$$
where the latter one is the restriction to $\HH_{\beta} 
\subset \G(\lambda)$, 
we regard the $\HH_{\beta}$-equivariant $K$-group
$K^{\HH_{\beta}}(\widetilde{\mathfrak{Z}}(\lambda; w_{\ii}))
\otimes_{A} \kk$
as a $U_{q}(L\g)$-module.

\begin{Thm}[Nakajima \cite{Nakajima01t}]
\label{Thm:tensor}
There is a $U_{q}(L\g)$-module isomorphism 
$$
K^{\HH_{\beta}}(\widetilde{\mathfrak{Z}}(\lambda; w_{\ii}))
\otimes_{A} \kk
\cong 
\bV^{\otimes \ii} :=
\bW(\lambda_{i_1}) \otimes \cdots
\otimes \bW(\lambda_{i_d}),
$$
where the action of $R(\HH_{\beta}) \otimes_{A} \kk$
on the LHS is translated into the action on the RHS 
via the isomorphism
\begin{align}
\label{Eq:isomO}
R(\HH_{\beta}) \otimes_{A} \kk
& \xrightarrow{\cong} \mathcal{O}_{\ii} := 
\kk[X_{1}^{\pm 1}, \ldots,  X_{d}^{\pm 1}] \subset 
\End_{U_{q}}(\bV^{\otimes \ii}); \\
[\C v_{w_{\ii}(k)}] & 
\mapsto
X_{k}, \nonumber
\end{align}
where we set
$X_{k}:= z_{\lambda_{i_k}}$
using the notation in (\ref{Eq:isomRfund}).
\end{Thm}

The decomposition (\ref{Eq:group})
$\G_{\beta} \cong G_{\beta} \times T_{\beta}$
induces the decomposition
$\HH_{\beta} \cong H_{\beta} \times T_{\beta}$
of the maximal torus $\HH_{\beta}$.
Similarly to the case of $\G_{\beta}$-equivariant $K$-groups 
in Subsection \ref{Ssec:Nakajima},
this decomposition yields 
a natural isomorphism
$$
K^{\HH_{\beta}}(X) 
\otimes_{A} \kk
\cong K^{H_{\beta}}(X)_{\kk}
$$
for any $\HH_{\beta}$-variety $X$ 
with a trivial $T_{\beta}$-action.
When $X = \mathrm{pt}$,
we have the following commutative diagram:
\begin{equation}
\label{Diag:R}
\xy
\xymatrix{
R(\HH_{\beta})\otimes_{A} \kk
\ar[r]^-{\cong}
\ar[d]^-{\cong}_{(\ref{Eq:isomO})}
&
R(H_{\beta})_{\kk}
\ar[d]^-{\cong}
\\
\mathcal{O}_{\ii}
=
\kk[X_{1}^{\pm 1}, \ldots, X_{d}^{\pm1}]
\ar[r]^-{\cong}
&
\kk[y_{1}^{\pm 1}, \ldots, y_{d}^{\pm 1} ] 1_{\ii},
}
\endxy
\end{equation}
where the bottom horizontal arrow
sends the element $a_{i_k}^{-1}X_{k}$
to $y_{k} 1_{\ii}$
for $1 \le k \le d$.
Under this isomorphism, the  
maximal ideal $\rr^{\prime}_{\beta} \subset 
R(\HH_{\beta}) \otimes_{A} \kk$
defined as the kernel of the restriction
$R(\HH_{\beta})\otimes_{A}\kk \to R(T_{\beta})\otimes_{A} \kk = \kk$
corresponds to the augmentation ideal
of $R(H_{\beta})_{\kk}$.
Therefore we have a natural isomorphism
\begin{equation}
\label{Eq:complH}
\left[ 
K^{\HH_{\beta}}(X) \otimes_{A} \kk
\right]^{\wedge}_{\rr^{\prime}_{\beta}}
\cong
\hK^{H_{\beta}}(X)_{\kk},
\end{equation}
where $[-]_{\rr^{\prime}_{\beta}}^{\wedge}$
denotes the $\rr^{\prime}_{\beta}$-adic completion.
In particular, completing the diagram (\ref{Diag:R}),
we get 
$$
\xy
\xymatrix{
\left[ R(\HH_{\beta})\otimes_{A} \kk 
\right]^{\wedge}_{\rr^{\prime}_{\beta}}
\ar[r]^-{\cong}
\ar[d]^-{\cong}
&
\widehat{R}(H_{\beta})_{\kk}
\ar[d]^-{\cong}
\\
\hO_{\ii}
:=
\kk[\![X_{1} -a_{i_1}, \ldots, X_{d}-a_{i_d}]\!]
\ar[r]^-{\cong}
&
\kk[\![y_{1} -1 , \ldots, y_{d}-1]\!] 1_{\ii}.
}
\endxy
$$ 

\begin{Thm}
\label{Thm:left}
We have the following isomorphism
of $U_{q}(L\g)$-modules:
$$
\widehat{\Phi}_{\beta}^{*}
\left(
\hK^{G_{\beta}}(\Mg_{\beta} \times_{E_{\beta}} \F_{\ii})_{\kk}
\right)
\cong
\hV^{\otimes \ii} :=
\bV^{\otimes \ii} \otimes_{\mathcal{O}_{\ii}} \hO_{\ii}.
$$
\end{Thm}
\begin{proof}
Actually, there is the following isomorphism:
\begin{align*}
\hV^{\otimes \ii} & \cong 
\left[
K^{\HH_{\beta}}(\tZ(\lambda; w_{\ii}))\otimes_{A} \kk
\right]_{\rr^{\prime}_{\beta}}^{\wedge}
&& (\text{Theorem~\ref{Thm:tensor}}) \\
& \cong \left[
K^{\HH_{\beta}}(\tZ(\lambda
; w_{\ii})^{T_{\beta}})\otimes_{A} \kk
\right]_{\rr^{\prime}_{\beta}}^{\wedge}
&& (\text{localization theorem})\\
& \cong \hK^{H_{\beta}}(\tZ(\lambda; w_{\ii})^{T_{\beta}})_{\kk}
&& (\text{isomorphism (\ref{Eq:complH})})\\
&\cong \hK^{G_{\beta}}(\Mg_{\beta} 
\times_{E_{\beta}} \F_{\ii})_{\kk}.
&& (\text{isomorphism (\ref{Eq:redK})})
\end{align*}
We need to show that this is
a $U_{q}(L\g)$-homomorphism.
By construction,
the following diagram 
of $\kk$-algebras commutes:
$$
\xy
\xymatrix{
K^{\G(\lambda)}(Z(\lambda))
\otimes_{A} \kk
\ar[r]
\ar[d]
&
\left[
K^{\G_{\beta}}(\Zg_{\beta})
\otimes_{A} \kk
\right]^{\wedge}_{\rr_{\beta}}
\ar[r]^-{\cong}_-{(\ref{Eq:complG})}
\ar[d]
&
\hK^{G_{\beta}}(\Zg_{\beta})_{\kk}
\ar[d]
\\
K^{\HH_{\beta}}(Z(\lambda)) 
\otimes_{A}\kk
\ar[r]
&
\left[
K^{\HH_{\beta}}(\Zg_{\beta})
\otimes_{A} \kk
\right]^{\wedge}_{\rr_{\beta}^{\prime}}
\ar[r]^-{\cong}_-{(\ref{Eq:complH})}
&
\hK^{H_{\beta}}(\Zg_{\beta})_{\kk},
}
\endxy
$$
where the vertical arrows denote 
the restrictions to the maximal tori.
Moreover, by using an $H_{\beta}$-equivariant version of
\cite[Proposition 8.2.3]{Nakajima01},
we can see that
the following diagram also commutes:
$$
\xy
\xymatrix{
K^{G_{\beta}}(\Zg_{\beta})_{\kk} \otimes 
K^{G_{\beta}}(\Mg_{\beta} \times_{E_{\beta}} \F_{\ii})_{\kk}
\ar[r]^-{*}
\ar[d]_{(\text{restriction to $H_{\beta}$}) \otimes (\ref{Eq:redK}) }  
&
K^{G_{\beta}}(\Mg_{\beta} \times_{E_{\beta}} \F_{\ii})_{\kk}
\ar[d]^-{\cong}_-{(\ref{Eq:redK})}
\\
K^{H_{\beta}}(\Zg_{\beta})_{\kk} \otimes 
K^{H_{\beta}}(\tZ(\lambda; w_{\ii})^{T_{\beta}})_{\kk}
\ar[r]^-{*}
&
K^{H_{\beta}}(\tZ(\lambda; w_{\ii})^{T_{\beta}})_{\kk},
}
\endxy
$$
where the horizontal arrows denote
the convolution products.  
From these commutative diagrams, combined with
the definition of $\widehat{\Phi}_{\beta}$
and Theorem~\ref{Thm:tensor},
we obtain the conclusion.
\end{proof}

\subsection{The right action of $\widehat{H}_{Q}(\beta)$}

Summarizing the discussion so far, 
we have obtained 
a $(U_{q}(L\g), \widehat{H}_{Q}(\beta))$-bimodule 
structure on the left $U_{q}(L\g)$-module
$$
\hV^{\otimes \beta} := \bigoplus_{\ii \in I^{\beta}}
\hV^{\otimes \ii}
$$
such that the following diagram commutes:
$$
\xy
\xymatrix{
U_{q}(L\g) 
\ar[r]
\ar[d]^-{\widehat{\Phi}_{\beta}}
&
\End(\hV^{\otimes \beta})
\ar[d]^-{\cong}
&
\hH_{Q}(\beta)^{\mathrm{op}}
\ar[l]_-{\exists \psi}
\ar[d]^-{\cong}
\\
\hK^{G_{\beta}}(\Zg_{\beta})_{\kk}
\ar[r]
&
\End \left( \hK^{G_{\beta}}(\Mg_{\beta} \times_{E_{\beta}} 
\F_{\beta})_{\kk}
\right)
&
\hK^{G_{\beta}}(\cZ_{\beta})_{\kk}^{\mathrm{op}}.
\ar[l]
}
\endxy
$$
In this subsection, we describe the right action 
$
\psi : \hH_{Q}(\beta) \to \End_{U_q}(\hV^{\otimes \beta})^{\mathrm{op}}
$
of the quiver Hecke algebra $\hH_{Q}(\beta)$
on the space $\hV^{\otimes \beta}$.

For each $\ii = (i_1, \ldots, i_d) \in I^{\beta}$, we set
$$
v_{\ii} := (w_{\lambda_{i_1}} \otimes \cdots \otimes 
w_{\lambda_{i_d}} )\otimes 1
\in \hV^{\otimes \ii} = (\bW(\lambda_{i_1}) \otimes 
\cdots \otimes \bW(\lambda_{i_d})) \otimes_{\mathcal{O}_{\ii}}
\hO_{\ii}.
$$

\begin{Prop}
\label{Prop:hwspace}
The highest weight space
$
\bigoplus_{\ii \in I^{\beta}} \hO_{\ii} v_{\ii} 
\subset \hV^{\otimes \beta} 
$
of weight $\lambda$
is stable under the right action of $\hH_{Q}(\beta)$.
Moreover it is isomorphic to 
the completed polynomial representation 
$\widehat{P}_{\beta}$ defined in (\ref{Eq:Phat}).
\end{Prop}
\begin{proof}
Note that the connected component
of the graded quiver variety
$\Mg_{\beta} = \M(\lambda)^{T_{\beta}}$
corresponding to the highest weight space 
is $\M(0, \lambda)^{T_{\beta}} = \mathrm{pt}$
and hence
$\M(0, \lambda)^{T_{\beta}} \times_{E_{\beta}} \F_{\beta}
= \B_{\beta}$.
Therefore we have
$$
\bigoplus_{\ii \in I^{\beta}} \hO_{\ii} v_{\ii} 
\cong \hK^{G_{\beta}}(\M(0, \lambda)^{T_{\beta}}
\times_{E_{\beta}} \F_{\beta})_{\kk}
\cong 
\hK^{G_{\beta}}(\B_{\beta})_{\kk} \cong \widehat{P}_{\beta}
$$
as $\hH_{Q}(\beta)$-module,
where the last isomorphism comes from
(\ref{Eq:isomP}) and (\ref{Eq:isomRRB}).
\end{proof}

Henceforth,
we normalize the isomorphism
$\hK^{G_{\beta}}(\Mg_{\beta} \times_{E_{\beta}} \F_{\ii})_{\kk}
\cong \hV^{\otimes \ii}$ of $U_{q}(L\g)$-modules
in Theorem~\ref{Thm:left}
by multiplying the element of $\hO_{\ii}$
corresponding to the ratio $C_{\ii}^{-1}$ of 
Todd classes defined in (\ref{Eq:const})
for each $\ii \in I^{\beta}$
so that
the isomorphism
$$
\bigoplus_{\ii \in I^{\beta}} \hO_{\ii} v_{\ii} 
= 
\bigoplus_{\ii \in I^{\beta}} \kk[\![ X_{1}-a_{i_1}, 
\ldots, X_{d}-a_{i_d} ]\!] v_{\ii} 
\xrightarrow{\cong}
\widehat{P}_{\beta} =
\bigoplus_{\ii \in I^{\beta}} \kk[\![x_{1}, \ldots, x_{d}]\!]1_{\ii}
$$
in Proposition~\ref{Prop:hwspace} above
sends the element $v_{\ii}$ to $1_{\ii}$. 

Now we recall the normalized $R$-matrices.
For any pair $(i_1, i_2) \in I^{2}$,
we simplify $z_{k} := z_{\lambda_{i_k}}$ for $k=1,2$.
Then it is known
(see e.g.~\cite[Section 8]{Kashiwara02}) 
that there is a unique
$(
U_{q} \otimes 
\kk[z_{1}^{\pm 1}, z_{2}^{\pm 1} ]
)$-homomorphism,
called the normalized $R$-matrix
$$
R_{i_1, i_2}^{\mathrm{norm}} :
\bW(\lambda_{i_1}) \otimes \bW(\lambda_{i_2}) 
\to
\kk(z_{2}/z_{1})\otimes_{\kk[(z_{2} / z_{1})^{\pm 1}]}
\left(
\bW(\lambda_{i_2}) \otimes \bW(\lambda_{i_1})
\right),
$$
such that $R^{\mathrm{norm}}_{i_1, i_2}
(w_{\lambda_{i_1}} \otimes w_{\lambda_{i_2}}) 
= w_{\lambda_{i_2}} \otimes w_{\lambda_{i_1}}$.
The denominator of
the normalized $R$-matrix
$R^{\mathrm{norm}}_{i_1, i_2}$
is defined as the monic polynomial 
$d_{i_1, i_2}(u) \in \kk[u]$
of the smallest degree 
among polynomials satisfying
$$
\Ima R^{\mathrm{norm}}_{i_1, i_2} \subset
d_{i_1, i_2}(z_{2} / z_{1})^{-1} \otimes 
\left(
\bW(\lambda_{i_2}) \otimes \bW(\lambda_{i_1})
\right).
$$
By \cite[Proposition~9.3]{Kashiwara02}, we have 
\begin{equation} 
\label{Eq:denominator}
d_{i_1, i_2}(1) \neq 0 .
\end{equation}

Let $\bK_{\ii}$
be the fraction field of the ring $\hO_{\ii}$
for each $\ii \in I^{\beta}$.
It is known that
the $U_{q} \otimes \bK_{\ii}$-module
$$
\hV^{\otimes \ii}_{\bK} :=
\bV^{\otimes \ii}\otimes_{\mathcal{O}_{\ii}} \bK_{\ii}
=\hV^{\otimes \ii} \otimes_{\hO_{\ii}} \bK_{\ii}$$
is irreducible (see e.g.~\cite[Proposition 9.5]{Kashiwara02}).
For each $w \in \SG_{d}$, the $\kk$-algebra isomorphism
$$
\varphi_{w}:
\hO_{\ii} \xrightarrow{\cong} \hO_{\ii \cdot w}; 
\quad
f(X_{1}, \ldots, X_{d}) \mapsto 
f^{w}(X_{1}, \ldots, X_{d}) := 
f(X_{w(1)}, \ldots, X_{w(d)})  
$$
induces an isomorphism 
$\bK_{\ii} \xrightarrow{\cong} \bK_{\ii \cdot w}$
of the fraction fields,
which we denote by the same symbol $\varphi_{w}$.
The pull-back
$
\varphi_{w}^{*} 
\hV_{\bK}^{\otimes \ii \cdot w} 
$
is an irreducible 
$U_{q}\otimes \bK_{\ii}$-module.

For each $\ii \in I^{\beta}$
and $1 \le k < d$,
we define the following non-zero
$U_{q}\otimes \bK_{\ii}$-homomorphism
$$
R^{\ii}_{k}
:=
\left(1^{\otimes (k-1)} \otimes 
R_{i_k, i_{k+1}}^{\mathrm{norm}}
\otimes 1^{\otimes(d-k-1)}
\right)
\otimes \varphi_{s_{k}}
: 
\hV_{\bK}^{\otimes \ii}
\to
\varphi_{s_{k}}^{*}
\hV^{\otimes \ii \cdot s_{k}}_{\bK}.
$$
By
the irreducibility,
this is an isomorphism and
we have
\begin{equation}
\label{Eq:hcKhom}
\Hom_{U_{q}\otimes \bK_{\ii}}
\left( \hV_{\bK}^{\otimes \ii},
\varphi_{s_{k}}^{*}
\hV^{\otimes \ii \cdot s_{k}}_{\bK} \right)
= 
\bK_{\ii} \cdot R^{\ii}_{k}.
\end{equation}
Let $\hV^{\otimes \beta}_{\bK} := \bigoplus_{\ii \in I^{\beta}} \hV^{\otimes \ii}_{\bK}$.
We regard 
$\hV^{\otimes \beta} \subset \hV^{\otimes \beta}_{\bK}$ naturally.

\begin{Thm}
\label{Thm:main}
The right action of the quiver Hecke algebra
$\hH_{Q}(\beta)$ on the space $\hV^{\otimes \beta}$ is given by 
the following formulas: 
\begin{align}
\label{Eq:e}
v \cdot e(\ii^{\prime}) & = \delta_{\ii, \ii^{\prime}} v \\ 
\label{Eq:x}
v \cdot x_{k} & = \log(a^{-1}_{i_k}X_{k}) v \\
\label{Eq:tau}
v \cdot \tau_{k} &=
\begin{cases}
\displaystyle
(
\log(a_{i_k}^{-1}X_{k}) - \log(a_{i_{k+1}}^{-1}X_{k+1}) )^{-1}
( R^{\ii}_{k}(v)- v )
&
\text{if $i_k = i_{k+1}$}, \\
(\log(a_{i_k}^{-1}X_{k+1}) - \log(a_{i_{k+1}}^{-1}X_{k})) R^{\ii}_{k}(v) &
\text{if $i_k \leftarrow i_{k+1}$}, \\
R^{\ii}_{k}(v) &
\text{otherwise},
\end{cases}
\end{align}
where 
$v \in \hV^{\otimes \ii}$
with $\ii = (i_{1}, \ldots, i_{d}) \in I^{\beta}$ and
$\log(X) :=
\sum_{m = 1}^{\infty} (-1)^{m+1}(X -1 )^{m}/m.
$
\end{Thm}
\begin{proof}
The first formula (\ref{Eq:e}) 
is clear from Theorem~\ref{Thm:VV}~(\ref{Thm:VV:elem}) 
and the construction.

To prove the second formula (\ref{Eq:x}), 
we assume that the vector $v \in \hV^{\otimes \ii}$ 
corresponds to an element $\zeta \in 
\hK^{G_{\beta}}(\Mg_{\beta} \times_{E_{\beta}} \F_{\ii})_{\kk}$
under the isomorphism in Theorem \ref{Thm:left}.
By Theorem~\ref{Thm:VV}~(\ref{Thm:VV:elem}),
the right action of $e^{x_{k}} \in \hH_{Q}(\beta)$ on $\hV^{\otimes \ii}$
corresponds to the convolution 
with the class $\Delta_{*}[\mathcal{L}_{\ii}(k)]
\in \hK^{G_{\beta}}(\F_{\ii} \times_{E_{\beta}} \F_{\ii})_{\kk}$
from the right, where
$\mathcal{L}_{\ii}(k)$ is the line bundle on $\F_{\ii}$ defined in Subsection 
\ref{Ssec:VV} and
$\Delta : \F_{\ii} \to \F_{\ii} \times_{E_{\beta}} \F_{\ii}$
is the diagonal embedding.
By \cite[Lemma 8.1.1]{Nakajima01}, 
we have
$
\zeta * (\Delta_{*}[\mathcal{L}_{\ii}(k)]) = \zeta \otimes p_{2}^{*}[\mathcal{L}_{\ii}(k)],
$
where $p_{2} : \Mg_{\beta} \times_{E_{\beta}} \F_{\ii} \to \F_{\ii}$
is the second projection.
The isomorphism (\ref{Eq:redK})  
translates the operation $- \otimes p_{2}^{*}[\mathcal{L}_{\ii}(k)]$
on $K^{G_{\beta}}(\Mg_{\beta} \times_{E_{\beta}} \F_{\ii})$
into the multiplication of the element $y_{k} 1_{\ii} \in R(H_{\beta})$
on $K^{H_{\beta}}(\tZ(\lambda; w_{\ii})^{T_{\beta}})$.
Thus we have 
$v \cdot e^{x_{k}} =  (a_{i_{k}}^{-1} X_{k})v$
(see (\ref{Diag:R})).

Let us verify the third formula (\ref{Eq:tau}).
Let $\psi : \hH_{Q}(\beta) \to \End_{U_{q}} ( \hV^{\otimes \beta})^{\mathrm{op}}$
be the structure morphism. 
First, we consider the case $i_k = i_{k+1}$.
From the commutation relation between $e(\ii) \tau_k$ and
$x_l$ in $H_{Q}(\beta)$, and the formula (\ref{Eq:x}) for $\psi(x_{l})$ 
which we have proved in the previous paragraph, 
we see that
$$
(\mathcal{D} \psi(e(\ii)\tau_{k}) + 1) f 
= f^{s_{k}}(\mathcal{D} \psi(e(\ii)\tau_{k}) + 1)
$$    
holds in $\End_{U_{q}}
(\hV^{\otimes \ii})$ for any $f \in \hO_{\ii}$,
where we put $\mathcal{D} := \log(a^{-1}_{i_{k}}X_{k}) - \log(a_{i_{k+1}}^{-1}X_{k+1})$.
In other words, the operator
$\mathcal{D} \psi(e(\ii)\tau_{k}) + 1$ belongs to
$
\Hom_{U_{q}\otimes \hO_{\ii} } 
( \hV^{\otimes \ii}, \varphi_{s_k}^{*}\hV^{\otimes \ii}).
$
Therefore it extends to
an operator on the localizations. Namely, we can regard
$$
\mathcal{D} \psi(e(\ii)\tau_{k}) + 1
\in \Hom_{U_{q} \otimes \bK_{\ii} } 
\left( \hV^{\otimes \ii}_{\bK}, \varphi_{s_k}^{*}\hV^{\otimes \ii}_{\bK}
\right) 
\cong \bK_{\ii} \cdot R^{\ii}_{k},
$$
where the last isomorphism is (\ref{Eq:hcKhom}).
By Proposition \ref{Prop:hwspace} and the formulas in Theorem 
\ref{Thm:KL}, we see that
$(\mathcal{D} \psi(e(\ii)\tau_{k}) + 1) v_{\ii} = v_{\ii} = R^{\ii}_{k}(v_{\ii})$.
Therefore we obtain
$\mathcal{D} \psi(e(\ii)\tau_{k}) + 1 = R^{\ii}_{k}$
as an operator on $\hV^{\otimes \ii}$.

The case $i_{k} \neq i_{k+1}$ is easier.
In this case, the commutation relation in $H_{Q}(\beta)$ 
and the formula (\ref{Eq:x}) for $\psi(x_{l})$ show that 
the operator $\psi(e(\ii) \tau_{k})$ already belongs to
$\End_{U_{q}\otimes \hO_{\ii}} ( 
\hV^{\otimes \ii}, \varphi_{s_{k}}^{*} \hV^{\otimes \ii})$. 
Therefore it extends to an element in
$\Hom_{U_{q} \otimes \bK_{\ii}}( \hV^{\otimes \ii}_{\bK}, 
\varphi_{s_k}^{*}\hV^{\otimes \ii}_{\bK}).$
Then we proceed just as in the previous paragraph
to
obtain the desired formula (\ref{Eq:tau}),
taking 
Proposition~\ref{Prop:hwspace}, the formulas 
in Theorem \ref{Thm:KL} and (\ref{Eq:hcKhom})
into consideration.
\end{proof}

\begin{Cor}[= \cite{KKK15} Conjecture 4.3.2]
\label{Cor:KKKconj}
For any $i_1, i_2 \in I$, the order of zero of the denominator
$d_{i_1, i_2}(u)$ at the point $u=a_{i_2} / a_{i_1}$ is at most one.
\end{Cor}
\begin{proof}
Since we know (\ref{Eq:denominator}), we may assume that
$i_1 \neq i_2$. We consider a sequence $\ii = (i_1, i_2) \in I^{\beta}$
with $\beta = \alpha_{i_1} + \alpha_{i_2}$.
When $i_{1} \leftarrow i_{2}$, the formula (\ref{Eq:tau})
tells us that the operator
$(\log(a_{i_1}^{-1}z_{1}) - \log(a_{i_2}^{-1}z_{2})) R^{\ii}_{1}$
belongs to $\Hom_{U_{q}}(\hV^{\otimes \ii}, \hV^{\otimes \ii \cdot s_{1}} )$, where we put 
$z_{k} = z_{\lambda_{i_k}}$ for $k=1,2$ as before. 
Notice that
$$
\log(a^{-1}_{i_1} z_{1}) - \log(a^{-1}_{i_2} z_{2}) \in 
( {z_2}/{z_1} - a_{i_2} / a_{i_1} )\cdot \hO_{\ii}^{\times}. 
$$
Therefore we find that the order of zero of $d_{i_1, i_2}(u)$
at $u = a_{i_2} / a_{i_1}$ is at most one.
For the other case $i_k \not \leftarrow i_{k+1}$, 
by the formula (\ref{Eq:tau}),
the operator $R^{\ii}_{1}$ already
belongs to $\Hom_{U_{q}} (\hV^{\otimes \ii}, \hV^{\otimes \ii \cdot s_{1}} )$.
Therefore the order of zero of $d_{i_1, i_2}(u)$
at $u = a_{i_2} / a_{i_1}$ is zero.
\end{proof}

\begin{Rem}
\label{Rem:KKK}
For each $\ii \in I^{\beta}$, we define a topological $\kk$-algebra automorphism
$\sigma_{\ii}$ of $\hO_{\ii}$ by setting
$$
\sigma_{\ii}(\log(a_{i_k}^{-1}X_{k})) := a_{i_k}^{-1}X_{k} -1
$$
for all $k$. This induces a $U_{q}(L\g)$-module automorphism 
$\sigma := \bigoplus_{\ii \in I^{\beta}} (1 \otimes \sigma_{\ii})$ on 
the module $\hV^{\otimes \beta}$. If we twist our right $\hH_{Q}(\beta)$-action
by this automorphism $\sigma$ (i.e.~we replace the structure map
$\psi$ with $\sigma\psi(-) \sigma^{-1}$),
we get a new right $\hH_{Q}(\beta)$-action
on $\hV^{\otimes \beta}$ given by the following formulas:
\begin{align}
\label{Eq:e2}
v \cdot e(\ii^{\prime}) & = \delta_{\ii, \ii^{\prime}} v \\ 
\label{Eq:x2}
v \cdot x_{k} & = (a_{i_k}^{-1}X_{k} -1 ) v \\
\label{Eq:tau2}
v \cdot \tau_{k} &=
\begin{cases}
\displaystyle
(
a^{-1}_{i_k}X_{k}  - a_{i_{k+1}}^{-1}X_{k+1})^{-1}
( R^{\ii}_{k}(v)- v )
&
\text{if $i_k = i_{k+1}$}, \\
(a_{i_k}^{-1}X_{k+1} - a_{i_{k+1}}^{-1}X_{k}) R^{\ii}_{k}(v) &
\text{if $i_k \leftarrow i_{k+1}$}, \\
R^{\ii}_{k}(v) &
\text{otherwise},
\end{cases}
\end{align}
where 
$v \in \hV^{\otimes \ii}$
with $\ii = (i_{1}, \ldots, i_{d}) \in I^{\beta}$.
This new action is same as
Kang-Kashiwara-Kim\rq{}s action in \cite{KKK18}, \cite{KKK15}.    
\end{Rem} 

\begin{Thm}
\label{Thm:equiv}
The formulas (\ref{Eq:e}), (\ref{Eq:x}) and (\ref{Eq:tau}) (or 
the formulas (\ref{Eq:e2}), (\ref{Eq:x2}) and (\ref{Eq:tau2}))
define a structure of a $(U_{q}(L\g), \hH_{Q}(\beta))$-bimodule
on the left $U_{q}(L\g)$-module $\hV^{\otimes \beta}$.
The functor $M \mapsto \hV^{\otimes \beta} \otimes_{\hH_{Q}(\beta)} M$
gives an equivalence of categories:
$$
\hH_{Q}(\beta) \modfd \xrightarrow{\simeq} \Cc_{Q, \beta}.
$$
\end{Thm}
\begin{proof}
This follows from the discussions in this subsection,
Theorem~\ref{Thm:mine} and Theorem~\ref{Thm:Morita}.
\end{proof}


\begin{thebibliography}{10}

\bibitem{CP96}
V.~Chari and A.~Pressley.
\newblock Quantum affine algebras and affine {H}ecke algebras.
\newblock {\em Pacific J. Math.}, 174(2):295--326, 1996.

\bibitem{CP01}
V.~Chari and A.~Pressley.
\newblock Weyl modules for classical and quantum affine algebras.
\newblock {\em Represent. Theory}, 5:191--223, 2001.

\bibitem{CG97}
N.~Chriss and V.~Ginzburg.
\newblock {\em Representation theory and complex geometry}.
\newblock Birkhauser Boston, Inc., Boston, MA, 1997.

\bibitem{EG00}
D.~Edidin and W.~Graham.
\newblock Riemann-{R}och for equivariant {C}how groups.
\newblock {\em Duke Math. J.}, 102(3):567--594, 2000.

\bibitem{Fujita17}
R.~Fujita.
\newblock Affine highest weight categories and quantum affine {S}chur-{W}eyl
  duality of {D}ynkin quiver types.
\newblock preprint. arXiv:1710.11288.

\bibitem{GRV94}
V.~Ginzburg, N.~Reshetikhin, and E.~Vasserot.
\newblock Quantum groups and flag varieties.
\newblock In {\em Mathematical aspects of conformal and topological field
  theories and quantum groups (South Hadley, MA, 1992)}, number 175 in Contemp.
  Math., pages 101--130. Amer. Math. Soc., Providence, RI, 1994.

\bibitem{Happel88}
D.~Happel.
\newblock {\em Triangulated categories in the representation theory of
  finite-dimensional algebras}, volume 119 of {\em London Mathematical Society
  Lecture Note Series}.
\newblock Cambridge University Press, Cambridge, 1988.

\bibitem{HL15}
H.~Hernandez and B.~Leclerc.
\newblock Quantum {G}rothendieck rings and derived {H}all algebras.
\newblock {\em J. Reine Angew. Math.}, 701:77--126, 2015.

\bibitem{KKK18}
S.-J. Kang, M.~Kashiwara, and M.~Kim.
\newblock Symmetric quiver {H}ecke algebras and {R}-matrices of quantum affine
  algebras.
\newblock {\it Invent. Math.}, 211(2):591--685, 2018.

\bibitem{KKK15}
S.-J. Kang, M.~Kashiwara, and M.~Kim.
\newblock Symmetric quiver {H}ecke algebras and {R}-matrices of quantum affine
  algebras, {II}.
\newblock {\it Duke Math. J.}, 164(8):1549--1602, 2015.

\bibitem{KKKO15}
S.-J. Kang, M.~Kashiwara, M.~Kim, and S.-j.~Oh.
\newblock Symmetric quiver {H}ecke algebras and {R}-matrices of quantum affine
  algebras, {III}.
\newblock {\it Proc. Lond. Math. Soc. (3)},111(2):420--444, 2015.

\bibitem{KKKO16}
S.-J. Kang, M.~Kashiwara, M.~Kim, and S.-j.~Oh.
\newblock Symmetric quiver {H}ecke algebras and {R}-matrices of quantum affine
  algebras, {IV}.
\newblock {\it Selecta Math. (N.S.)}, 22(4):1987--2015, 2016.

\bibitem{Kashiwara94}
M.~Kashiwara.
\newblock Crystal bases of modified quantized enveloping algebra.
\newblock {\em Duke Math. J.}, 73(2):383--413, 1994.

\bibitem{Kashiwara02}
M.~Kashiwara.
\newblock On level-zero representations of quantized affine algebras.
\newblock {\em Duke Math. J.}, 112(1):117--175, 2002.

\bibitem{Kato14}
S.~Kato.
\newblock {P}oincare-{B}irkhoff-{W}itt bases and {K}hovanov-{L}auda-{R}ouquier
  algebras.
\newblock {\em Duke Math. J.}, 163(3):619--663, 2014.

\bibitem{KL09}
M.~Khovanov and A.~Lauda.
\newblock A diagrammatic approach to categorification of quantum groups. {I}.
\newblock {\em Represent. Theory}, 13:309--347, 2009.

\bibitem{Nakajima01}
H.~Nakajima.
\newblock Quiver varieties and finite-dimensional representations of quantum
  affine algebras.
\newblock {\em J. Amer. Math. Soc.}, 14(1):145--238, 2000.

\bibitem{Nakajima01t}
H.~Nakajima.
\newblock Quiver varieties and tensor products.
\newblock {\em Invent. Math.}, 146(2):399--449, 2001.

\bibitem{OS18}
S.-j. Oh and T.~Scrimshaw.
\newblock Categorical relations between {L}anglands dual quantum affine
  algebras: {E}xceptional cases.
\newblock preprint. arXiv:1802.09253.

\bibitem{Rouquier08}
R.~Rouquier.
\newblock 2-{K}ac-{M}oody algebras.
\newblock preprint. arXiv:0812.5023.

\bibitem{VV11}
M.~Varagnolo and E.~Vasserot.
\newblock Canonical bases and {KLR}-algebras.
\newblock {\em J. Reine Angew. Math.}, 659:67--100, 2011.

\end{thebibliography}
\end{document}